\documentclass[12pt]{amsart}
\usepackage[latin1]{inputenc}
\usepackage{amsmath, latexsym, amsfonts, amssymb, amsthm, amscd}
\usepackage[left=3cm, right=2.5cm, top=2.5cm, bottom=2.5cm]{geometry}
\usepackage{color}

\newtheorem{corollary}{Corollary}
\newtheorem{proposition}{Proposition}
\newtheorem{lemma}{Lemma}

\newtheorem{theorem}{Theorem}

\renewcommand{\P}{\mathbb{P}}

\newcommand{\R}{\mathbb{R}}

\newcommand{\G}{\mathbb{G}}
\newcommand{\V}{\mathbb{V}}
\newcommand{\T}{\mathcal{T}}
\newcommand{\A}{\mathcal{A}}
\newcommand{\N}{\mathbb{N}}
\newcommand{\X}{\mathcal X}

\newcommand{\Ind}[1]{\mathbf{1}_{\{#1\}}}
\newcommand{\E}{\mathbb{E}}

\def\Bea{\begin{eqnarray*}}
\def\Eea{\end{eqnarray*}} 
\def\bea{\begin{eqnarray}}
\def\eea{\end{eqnarray}} 

\def\cmv#1{\marginpar{\raggedright\tiny{\textcolor{white}{Vincent:}  \textcolor{white}{#1}}}}
\usepackage{tikz}

\begin{document}

 \title[]{Spine for  interacting populations and sampling}

\author{Vincent Bansaye}
\address{CMAP, \'Ecole Polytechnique, Route de Saclay, F-91128 Palaiseau Cedex, France}
\email{vincent.bansaye@polytechnique.edu}

 \maketitle
%





\begin{abstract} 
We consider Markov jump processes describing structured populations with 
interactions via density dependance. We propose a Markov construction with a distinguished individual which allows to describe
      the random tree and random sample at a given time  via a change of probability. This spine construction involves the extension of type space of individuals to include the state of the population.
 The jump rates  outside the spine are also modified.
We exploit this approach to  study  some issues concerning population dynamics. For single type populations, we derive the diagram phase of a 
 growth fragmentation model with competition and the growth of the size of transient birth and death processes with multiple births.
We also describe the  ancestral lineages of a uniform sample in  multitype populations. 
\end{abstract}


\section{Introduction}
Spine techniques and size biased trees have a long and fruitful story in branching processes.
They have played a deep role in  the analysis of branching  brownian motion and branching random walk  from the works of Chauvin and Rouault \cite{CR} and
 Lyons \cite{Lyons}. More generally,    spine construction provides a relevant point of view to tackle many issues related to the genealogy and long time behavior of branching processes.
 Lyons, Peres and Pemantle  \cite{LPP} have given a conceptual approach of the famous $L\log L$ criterion involved in the asymptotic analysis of branching processes. Their construction provides an illuminating
proof of the non degenerescence of limiting martingale using  branching processes with immigration.
This Markov construction has been extended to multitype branching process \cite{KLPP}
and    infinite dimension and we refer e.g. to \cite{Athreya, Englander, EHK}. It involves then either an eigenfunction or an exponential additive functional of trajectories corresponding to a Feynman Kac semigroup, both being linked by a change of probability.  This construction is  also involved in the fine estimate of the front 
 of branching random walks, see   Hu and Shi  \cite{HS} and Roberts \cite{Roberts}, thanks to a family of  exponential eigenfunctions
 and Biggins martingale. \\
 Spine constructions give a trajectorial and markovian sense to a typical individual alive at a given time. This  allows to prove   ergodic properties of type distribution among the population and shed light on  sampling \cite{Marguet}. More generally, we refer to the description of  reduced tree and backbone  \cite{HHK, EKW} and multispine construction  \cite{HJR}.   It has finally proved to be a powerful way to analyse the first moment semigroup of branching processes, and more generally non-conservative semigroups or linear  PDEs, see e.g.
\cite{BCGM, BW} and references therein.  \\

These  constructions exploit  the branching property and independence of individuals.
 The
aim of this paper is to propose a spine construction for   dynamics taking into account interactions, through density dependance of individual behaviors and associated \emph{branching events}. Individuals may die  or reproduce or move and individual branching events may depend  on the state of the population. 
Such models are involved in population dynamics or genetics or epidemiology, see e.g. \cite{Kurtz, BM, EK} and forthcoming examples. Indeed, for various models of competition, mutualism, contamination, sexual reproduction or predation,   the individual rate of transition
 depends on the size of some species at a given location. Death or successive reproduction of asexual population may indeed depend on local competition and ressources available, reproduction of females may depend on density of males, contaminations by  infected individuals depends on the local number of susceptibles...  Let us also mention that density dependent models appear in various other contexts, including  chemistry, queueing systems or networks.\\
 Our first motivation here is the study of   population models with competition. The evolution of the distribution of traits among such a population is related to the distribution of a sample. They can both  be tackled via a spine construction. Addario-Berry and Penington \cite{AP}
 have considered a competitive effect in a branching random walk. The authors obtain fine  results on the front of propagation. They   focus  on a peculiar form of competition which enables them to link their model to branching brownian motion.  In a large population approximation where a branching property can be recovered, Calvez, Henry, M\'el\'eard and Tran  \cite{CHMT} describe the ancestral lineage of   a sample in a context of competition and adaptation to a gradual environmental change. The examples considered in this work will involve different scalings.  \\

In recent decades, lots of attention has been paid to the study of genealogical structures of population. For branching processes, the contour (or exploration) process provides a full description of the genealogy. From the work of Aldous and convergence to the continuum random tree, it has been  generalized and used for instance for the description of mutations of splitting trees \cite{Lambert}. The effect of competition as a pruning of trees has been introduced and studied  in  \cite{PW, BFF}.  Spine construction offers a complementary insight in the structure by focusing on a \emph{typical individual} in various senses. It allows for extension to structured population and varying environment.
An other point of view enlightens genealogical structure of 
population models, the look-down construction  introduced by Kurtz and Donnelly. In this construction, a level is added to individuals.  The Poisson representation of this enlarged process allows in particular to build the genealogy of large population approximations and describe the longest branch in the tree. We refer to \cite{KR} for the look-down construction of branching processes and to  \cite{EK} for a recent extension to interacting populations. This latter allows for a description of genealogy and samples by  a tracing  which 
follows the evolution of the levels back in time. We  consider in this paper simpler models and propose a forward Markov consistent construction for samples.\\

The main objective is the study of structured populations : each individual has a type $x\in \mathcal X$, where 
$\mathcal X$ is finite or countable  here. The type can represent 
a size, a location, or any phenotypic or genotypic  trait of the individual.
The population is described by a vector ${\bf z}=({\bf z}_x :  x\in \mathcal X)$ where ${\bf z}_x$ is the number of individuals with type $x$. We write 
$\| {\bf z}\|_1=\sum_{x\in \mathcal X} {\bf z}_x$ the $\ell_1$ norm of   ${\bf z}$
and work with the associated normed  and  countable space 
$$\mathcal Z=\{ {\bf z} \in  \mathbb N_0^{\mathcal X} :\|{\bf z}\|_1<\infty\}.$$
Informally,   each individual of  a   population composed by ${\bf z}$ branches independently and each individual with type $x$ is replaced by an  offspring
 ${\bf k}\in \mathcal Z$ with rate $\tau_{\bf k}(x,{\bf z})$. In other words,
 an individual with type $x$ branches at total rate $$\tau(x,{\bf z})=\sum_{{\bf k}\in\mathcal Z} \tau_{\bf k}(x,{\bf z})<\infty$$
 and is  replaced by ${\bf k}$ with probability $p_{\bf k}(x,{\bf z})= \tau_{\bf k}(x,{\bf z})/ \tau(x,{\bf z})$. The new composition of the population is then  ${\bf z}-{\bf e}(x)+{\bf k}$, where 
 ${\bf e}(x)$ stands for one single individual with trait $x$, i.e.
 ${\bf e}(x)=({\bf e}_y (x), \,  y\in \mathcal X)$ and  ${\bf e}_x(x)=1$ and ${\bf e}_y (x)=0$ for  $y\ne x$. For a reference on density dependent Markov process, let us mention \cite{EKbook, Kurtz}.\\

The spine construction consists in a new process with a distinguished individual and rates are modified using a positive function $\psi$ on $\X\times \mathcal Z$.
Roughly, when the  distinguished individual has type $x$ and lives
in population whose  type composition is ${\bf z}$,
this individual branches, yields ${\bf k}$ offsprings and
switches to type $y$ with rate
$$\widehat{\tau}_{\bf k}^{\, \star}(x,{\bf z})=\tau_{\bf k}(x,{\bf z}) \frac{ {\bf k}_y\, \psi(y,{\bf z}-{\bf e}(x)+{\bf k})}{ \psi(x,{\bf z})}.$$ 
  This rate
 is biased by the size and type of offsprings as for branching structures.
 It is also   corrected by the variation of the population composition  
 through a $\psi$ transform. The jump rates of individuals outside the spine are also modified and they  branch at rate
$$\widehat{\tau}_{\bf k}(y,x,{\bf z})=\tau_{\bf k}(y,{\bf z}) \frac{\psi(x,{\bf z}-{\bf e}(y)+{\bf k})}{ \psi(x,{\bf z})}$$ when their type is $y$ and the type  of the spine is $x$.
We observe that in the case when $\psi$ is not dependent on the state ${\bf z}$ of the population, we recover the construction for branching structures proposed in \cite{LPP,KLPP}.
\\
In the applications considered here, the couple formed by  the typical (or spine) individual and the composition of the population is involved. We will also consider the associated semigroup and martingale. The $\psi$-spine construction  
couples a size bias  and a  Doob  transform  on the product space $\mathcal X\times \mathcal Z$, which will provide a Feynman Kac representation of the semigroup. The spine will describe the lineage of a typical individual, which includes its time of branching, number of offsprings, types...  More information on the underlying genealogy structure may be interesting.  In particular, the tree associated to a sample is needed
when tracing an infected individual in epidemiology or when
looking at the subpopulation carrying a common mutation in
 population genetics.
Our main result   allows to describe  the full tree around the spine. 
Among stimulating open questions is the way multisampling could be obtained,  which will be just briefly evoked here.\\
We focus in this paper on the continuous time setting. The  spine construction achieved  has a counterpart in discrete time, for non-overlapping generations. As far as we see,
the fact that in continuous time branching  events are not simultaneous is actually more convenient for construction and analysis. Besides, models which motivate this work may be 
 more classical or relevant in continuous time. \\

{\bf Example.}
To motivate this construction and illustrate it, let us briefly present  the non-structured case
and refer to Section 3.1 for details. In that case, individuals are exchangeable, i.e. each individual has a common type. When the population size is $z$,  each individual branches and is replaced by $k$ individuals at rate $\tau_k(z)$. In other words, an event occurs inside the population at rate $z\tau(z)$ and then one individual is chosen uniformly at random and is replaced by $k$ individuals with probability $\tau_k(z)/\tau(z)$.
Let us consider the case where the jump Markov process $Z$ on $\N$ counting the number of individuals is well defined for any time (non-explosive) and does become extincted : $\forall t\geq 0, Z(t)\in [1,\infty)$. At a given time $t$, sample uniformly at random one individual alive. Then the times when the ancestral line has branched and the number of offsprings at these times is given (in law) by the $1/z$-spine construction. It means that it can be constructed in the forward sens by considering a Markov process with a distinguished individual (our sample) which  is replaced by $k$ individuals  at rate $k\tau_k(z)(z+k-1)/z$, while the other individuals branch independently and are  replaced by $k$ individuals at rate $\tau_k(z)(z+k-1)/z$. In particular, this result allows to specify how sampling bias  the reproduction of individuals along time and the effect of population size.\\

  {\bf Outline of the paper.}
  The paper is organized as follows. In the next section,  we describe more precisely the $\psi$-spine construction associated to a positive function $\psi$. The main result provides then a Girsanov type result (change of probability) to transform the original random tree with a randomly chosen individual at a give time into a new random tree with a distinguished individual, the spine. We complement this section by considering the associated semigroup and martingale, and a many-to-one formula, which focuses on the ancestral lineage of a typical individual.
The two next sections are devoted to  applications. In Section \ref{dim1}, we consider the single type case. In that case, computations can be achieved. It allows in particular  to describe 
explicitly the uniform sampling at a given time when extinction does not occur.
We exploit and illustrate this construction by considering a simple growth fragmentation process with competition and we determine the criterion of regulation of growth by competition and fragmentation. 
We also provide in this section a $L\log L$ criterion for the non-degenerescence   
of the natural positive martingale associated with the growth of the process, thus extending the criterion of Kesten Stigum and the approach of \cite{LPP}.  In Section \ref{multitype}, 
we consider a population structured by a finite number of types. We describe the ancestral lineage of a uniform sample  when the state space of the population is finite  and the sampling in large population approximation when the limiting process is a differential equation.\\

 In the paper, we write $\N=\{1,2,\ldots\}$, $\N_0=\{0,1,2,\ldots\}=\N\cup\{0\}$. For two vectors ${\bf u}=({\bf u}_x)_{x\in \X}$ and ${\bf v}=({\bf v}_x)_{x\in \X}$, we write
 $\langle {\bf u}, {\bf v}\rangle=\sum_{x\in \X} {\bf u}_x {\bf v}_x$ the inner product.
 \section{Spine construction}
 \label{mainsection}
 \subsection{Definition of the original process}
 \label{model}
We construct the tree of individuals  with their types, until the potential explosion time, as follows. We use the  Ulam Harris Neveu notation to label  the individuals of the population  and each label  has a type and life length.
We thus introduce
$$\mathcal U= \cup_{k\geq 1} \N^k, $$
where $u=(u_1,\ldots,u_k)\in  \N^k$ means that $u$ is  
an individual of the generation $k=\vert u\vert $ and the $u_{k}$-th child  of $(u_1,\ldots, u_{k-1})$.
We consider now a random  process $Z$ and the associated random tree $\mathcal T$
 tree   defined by iteration.
We start  with a population labeled by a non-empty and  finite and  deterministic subset $\mathfrak g$ of $\N$
and their types  are $(x_{u},\,  u\in \mathfrak g)$.  We write
$${\bf x}=\{(u,x_u), \, u \in \mathfrak g\}$$
this initial condition.
We denote by  ${\bf v}\in \mathcal Z$ the vector counting the initial number of individuals of each type:  ${\bf v}_x=\#\{ u \in \mathfrak g : x_u=x\}$.\\
The population alive at time $t$ is a random subset of $\mathcal U$, denoted by $\mathbb G(t)$, and the types of individuals at time $t$ are $(Z_{u},\,  u \in \mathbb G(t))$. The vector counting 
the number of individuals of each type is 
${\bf Z}(t)=({\bf Z}_x (t), \,   x\in \X)$, where ${\bf Z}_x (t)=\#\{ u \in \mathbb G(t) : Z_u=x\}$.
In particular,  ${\bf Z}(0)={\bf v}$.\\
The markovian construction by iteration is classical. Each individual $u$ has a random life length $L_u\in (0,+\infty]$
and a type $Z_u$ during all its life. Each individual with type $x$ is replaced by an  offspring
 whose types are counted by ${\bf k}\in \mathcal Z$ at rate $\tau_{\bf k}(x,{\bf z})$ when the population composition
 is ${\bf z}$.  When an individual $u\in \mathcal U$ is replaced, the new individuals are labeled by $(u,1), \ldots, (u,\| {\bf k} \|_1)$ and the population composition moves to ${\bf z}+{\bf k}-e(x)$. Each new individual has a type and the order of affectation of types
 will play no role. Let us  prescribe a type $Z_{(u,i)}$ to each new label $(u,i)$, for $1\leq i\leq \| {\bf k} \|_1$, and will need do it in coherent way later in the spine construction. Thus, we consider a probability law $Q_{\bf k}$ on 
$$\X_{{\bf k}}=\{{\bf x} \in \X^{\| {\bf k} \|_1} : \, \forall x \in \X, \, \#\{ i \geq 1 : {\bf x}_i=x\}={\bf k}_x\}$$ and 
$(Z_{(u,i)} : 1\leq i \leq \| {\bf k} \|_1)$  is distributed as $Q_{\bf k}$. This affectation  is
achieved independently for each event and its law only depends on the type composition ${\bf k}$. A generic natural law  
is an exchangeable one, choosing successively  the types of individuals uniformly at random among available choices, but models may suggest another one.

The process is constructed iteratively. Writing $T_n$ the successive branching times, the process is   constant in the time intervals $[T_n,T_{n+1})$, where $T_0=0$ and $T_{n+1}=+\infty$ if no event occurs after 
$T_n$. At these times $T_n$, we may say \emph{jump} or \emph{event} or  \emph{branching event}, indifferently. Note that  for any event, only one individual disappears. It may be replaced
by a single individual with a same type (but a different label).
The process is thus well defined until  the limiting time
of successive branching events $(T_n)_{n\geq 1}$: 
$$T_{\text{Exp}}=\lim_{n\rightarrow\infty}T_n\in \R\cup\{+\infty\}.$$
This latter is finite if the sequence of branching events accumulate and as usual, we speak then of \emph{explosion}. 
We write $\mathcal T$ the random tree obtained with this construction and 
$\mathcal T(t)$ the tree truncated at time $t\geq 0$.  Formally $\mathcal T=\{ (u, L_u,Z_u) : 
u \in \mathcal U, \, \exists t \geq 0 \text{ s.t.} \,  u \in \mathbb G(t)\}$ and  $\mathcal T(t)=\{ (u, L_u(t),Z_u) : 
u \in \mathcal U, \, \exists s \leq t \text{ s.t.} \, u \in \mathbb G(s)\}$ with $L_u(t)$ the 
life length of $u$ truncated  at time $t$ .

\subsection{The $\psi$-spine construction} Recall that $\mathcal Z=\{ {\bf z} \in  \mathbb N_0^{\mathcal X} :\|{\bf z}\|_1<\infty\}$ is  denumerable and gives the state space
of the composition of the population. We introduce now the state space for the type of the spine and composition of the population:
$$\overline{\mathcal Z}= \{ (x,{\bf z}) \in  \mathcal X\times \mathcal Z : {\bf z}_x\geq 1\}$$ and consider 
a (fixed) positive  function 
$$\psi: \overline{\mathcal Z}\rightarrow (0,\infty).$$
 We assume in the rest of the paper that for any
$(x,{\bf z})\in  \overline{\mathcal Z}$,
\begin{align}
\label{condsum}
&\sum_{{\bf k}\in \mathcal Z} \tau_{{\bf k}}(x,{\bf z}) \langle {\bf k},\psi(.,{\bf z}+{\bf k}-{\bf e}(x))\rangle <\infty.
\end{align}
Let us construct a new process ${\bf \Xi}^{\psi}={\bf \Xi}$
  and associated genealogical tree $\mathcal A^{\psi}=\mathcal A$. This construction contains   a distinguished individual  $E^{\psi}=E$ for any time. We follow the point of view of \cite{LPP,KLPP} for Galton-Watson processes. 
 We write now  $\mathbb V(t)$ the  random set of individuals alive at time $t$  
and the types of individuals are given by $(\Xi_u, \, u\in  \mathbb V(t))$.
Thus  $ E(t) \in \mathbb V(t)
\subset \mathcal U$ is the label of the spine at time $t$ and  the type of the spine  is then $Y(t)=\Xi_{E(t)}$.  Besides, ${\bf \Xi}_x (t)=\#\{ u \in \mathbb V(t) : \Xi_u=x\}$.

We start with the same  population ${\bf x}=\{(u,x_u), \, u \in \mathfrak g\}$  as  the original process, i.e. initial individuals are labeled by
$\mathfrak g$ and have types $(x_u : u \in \mathfrak g)$  counted by ${\bf v}$. 
The distribution of the initial label of the spine $E(0)$ is 
$$\P(E(0)=e)= \frac{\psi(x_e, {\bf v})}{\langle {\bf v},\psi(., {\bf v})\rangle} \qquad (e\in \mathfrak g). $$
Then  the distribution of the initial type of the spine is   $\P(Y(0)=r)= {\bf v}_r\psi(r, {\bf v})/ \langle {\bf v},\psi(., {\bf v})\rangle.$\\
$\newline$
Among a population whose types are counted by ${\bf z}$, the spine $E$ with type $x$ branches and  is replaced by  offsprings of types ${\bf k}$  at rate 
$$\widehat{\tau}_{\bf k}^{\, \star}(x,{\bf z})=\tau_{\bf k}(x,{\bf z}) \frac{\langle {\bf k},\psi(.,{\bf z}-{\bf e}(x)+{\bf k})\rangle}{ \psi(x,{\bf z})}.$$
The total branching  rate  of  the spine individual is then
 $\widehat{\tau}^{\, \star}(x,{\bf z})=\sum_{{\bf k}\in \mathcal Z}  \widehat{\tau}_{{\bf k}}^{\, \star} (x,{\bf z})$, which is finite by $\eqref{condsum}$.
Labels of offsprings are  $(E(t-),1),\ldots, (E(t-),\| {\bf k} \|_1)$ and their types
are chosen using probability law $Q_{{\bf k}}$ as above. Among these offsprings, each individual with type $y\in \mathcal X$ is chosen to be the spine
with probability
$$q_y({\bf k},{\bf z})=\frac{ \psi(y,{\bf z}-{\bf e}(x)+{\bf k})}{\langle{\bf k},\psi(.,{\bf z}-{\bf e}(x)+{\bf k})\rangle}$$
and provides the new label $E(t)$ of   the  distinguished individual.\\
Outside the spine, i.e. for individuals $u\in \mathbb V(t)-\{E(t)\}$ at time $t$, rates of jumps are  modified as follows.
 Inside a population  ${\bf z}$ with spine of type $x$, the individuals (but  the spine) with type $y$
 branch  and yield offsprings composed by ${\bf k}$ at rate 
$$\widehat{\tau}_{\bf k}(y,x,{\bf z})=\tau_{\bf k}(y,{\bf z}) \frac{\psi(x,{\bf z}-{\bf e}(y)+{\bf k})}{ \psi(x,{\bf z})}.$$
This process with a distinguished particle is constant between successive jumps $\widehat{T}_n$ and $\widehat{T}_{n+1}$, where $\widehat{T}_0=0$ and $\widehat{T}_{n+1}=+\infty$ if no event occurs after 
$\widehat{T}_n$. It is thus also defined by induction until  explosion time 
$$\widehat{T}_{\text{Exp}}=\lim_{n\rightarrow \infty} \widehat{T}_n\in \R_+\cup\{+\infty\},$$ which may be finite or not.\\
The Markovian construction achieved here provides a random tree $\mathcal A$ with a distinguished individual $E$. It  is associated to the original random process
through the rates $(\tau_{\bf k}(x,{\bf z}) :  x\in \mathcal X, {\bf k}\in \mathcal Z, {\bf z}  \in \mathcal Z)$ and the initial type composition ${\bf v}$. It then depends only
 on the choice of the transform $\psi$, which will play a key role.

\subsection{General result}
We introduce 
the linear operator $\mathcal G$ which arise in the spine construction.
For a  function  $f: \overline{\mathcal Z}\rightarrow \R$, we consider  the function
$\mathcal Gf$ on $\overline{\mathcal Z}$ given by
\Bea
\mathcal Gf(x,{\bf z})
&=& \sum_{{\bf k}\in \mathcal Z} \tau_{{\bf k}}(x,{\bf z}) \, \langle {\bf k},f(.,{\bf z}+{\bf k}-{\bf e}(x))\rangle\\
&&+\sum_{y \in \mathcal X,{\bf k} \in \mathcal Z} \tau_{{\bf k}}(y,{\bf z})({\bf z}_y-\delta_y^x)
f(x,{\bf z}+{\bf k}-{\bf e}(y)) -\left(\sum_{y \in \mathcal X}\tau(y,{\bf z}){\bf z}_y\right)f(x,{\bf z}),
\Eea
 where $\delta_y^x$ is the Kronecker symbol ($\delta_y^x=1$ if $y=x$ and $0$ otherwise).  This operator $\mathcal G$ is well defined and will be used
  on the set
$\mathfrak D_\mathcal G$  of positive functions $\psi$ on  $\overline{\mathcal Z}$
which satisfy \eqref{condsum}
and for any $(x,{\bf z})\in \overline{\mathcal Z}$,
$$\sum_{y \in \mathcal X,{\bf k} \in \mathcal Z} \tau_{{\bf k}}(y,{\bf z}){\bf z}_y
\psi(x,{\bf z}+{\bf k}-{\bf e}(y))<\infty.$$
In particular,  $\mathfrak D_\mathcal G$
  contains all the  bounded functions on  $\overline{\mathcal Z}$, which satisfy \eqref{condsum}.
By now, we assume that $\psi$ belongs to
$\mathfrak D_\mathcal G$
and we  define the real valued function 
$$  \qquad   \qquad  \qquad  \lambda=\frac{ \mathcal G\psi}{\psi}   \qquad  \text{on } \quad \overline{\mathcal Z}.\qquad   \qquad$$
Observe that the $\psi$-transform $f\rightarrow \mathcal G(\psi f)/\psi-\lambda f$ 
 yields the generator
of $(Y(t), {\bf \Xi}(t))_{t\geq 0}$.
For any $t\geq 0$, we consider  a random variable $U(t)$ choosing an individual alive at time $t$ among the original population process, when the population is alive. 
Its law is specified by the function $p_e(\mathfrak t)$ which yields the probability
to choose $e$ when the tree is $\mathfrak t$, i.e.
 $$\P( U(t)=e\, \vert\,  \mathcal T(t))=p_e(\mathcal T(t))$$
and $\sum_{e\in \mathbb G(t)} p_e(\mathcal T(t))=1$ a.s. on the event $\mathbb G(t)\ne \varnothing$. Our main interest in this paper is the uniform choice at time $t$, i.e. $\P( U(t)=e\, \vert\,  \mathcal T(t))=p_e(\mathcal T(t))=1/
\# \mathbb G(t)=1/ \| {\bf  Z}(t) \|_1$. But sampling at a given time with a type bias is also relevant for instance.\\
We introduce the  random process $\mathcal W$ associated with the spine construction
$(\mathcal A, E)$  and the  choice $p$:
 $$\mathcal W(t)
 ={\bf 1}_{\widehat{T}_{\text{Exp}}>t}\, \frac{\exp\left(\int_0^t \lambda(Y(s),{\bf \Xi}(s))ds \right)}{\psi(Y(t), {\bf \Xi}(t))} \, p_{E(t)}(\A(t)).$$
We can now state the result and link the random choice of an individual 
among our interacting population to the Markovian spine construction.
Let $\mathbb T$ be the space of finite trees where each nodd has a life length and an $\X$ valued type.
Elements of $\mathbb T$ are   identified to a finite collection of elements of 
$\mathcal U\times (\R_+\cup\{+\infty\} ) \times \X$ 
endowed with the product $\sigma$-algebra. 
\begin{theorem} \label{main} Let 
$\psi\in \mathfrak D_{\mathcal G}$. 
For any $t\geq 0$ and   any measurable non-negative function $F : \mathbb T\times \mathcal U\rightarrow \R$ :
$$\E_{{\bf x}}\left( \Ind{T_{\emph{Exp}}>t, \, \mathbb G(t)\ne \varnothing} \, F(\T(t),U(t))\right)=\langle  {\bf v},\psi(., {\bf v})\rangle\, \E_{{\bf x}}\left(\mathcal W(t) \, F(\A(t), E(t))\right).$$
\end{theorem}
In particular,
if $U(t)$ is a uniform choice among the set $\mathbb G(t)$ of individuals  at time $t$,
\begin{align*}
 &\E_{{\bf x}}\left( \Ind{T_{\text{Exp}}>t, \mathbb G(t)\ne \varnothing} \, F(\T(t), U(t))  \right)\\
 &\qquad =\langle  {\bf v},\psi(., {\bf v})\rangle \, \E_{{\bf x}}\left({\bf 1}_{\widehat{T}_{\text{Exp}}>t}\, \frac{e^{\int_0^t \lambda(Y(s),{\bf \Xi}(s))ds}}{\psi(Y(t), {\bf \Xi}(t)) \, \| {\bf  \Xi}(t) \|_1} 
F(\A(t),E(t))\right).
\end{align*}

The proof is a consequence of the following lemma, which encodes the successive
transitions. Recall that  the successive branching times of the original process ${Z}$
and of  the spine process ${\Xi}$  are respectively denoted by  $(T_i, 1\leq  i\leq  N)$, with $N\in \N \cup\{+\infty\}$  and ($\widehat{T}_i, \, 1\leq  i \leq  \widehat{N})$, with $\widehat{N}\in \N \cup\{+\infty\}$ and  $T_0=\widehat{T}_0=0$ a.s. The variable $N \in \mathbb N\cup\{+\infty\}$ yields the total number of branching events
 and $N=i<\infty$ means 
that  the process does not branch after time $T_i$. The same holds for $\widehat{N}$.   \\ 
For $1\leq i\leq N$, we write  $U_i$  (resp. ${\bf K}_i$) the random variable in $\mathcal U$ 
(resp. in $\mathcal Z$) which gives the label of the individual which realizes the 
$i$th branching  events in the original process (resp. the types of its offsprings at this event). We  denote by $(X_{i,j}, j\leq  \|{\bf K}_i\|_1)$ the types of the successive offsprings of $U_i$. In other words, at time $T_i$, the individual $U_i$ is replaced by individuals $(U_i,j)$, for $1 \leq j\leq  \|{\bf K}_i\|_1$, whose types are   $(X_{i,j}, 1\leq j\leq  \|{\bf K}_i\|_1)$.\\
We write similarly $\widehat{U}_i$, $\widehat{{\bf K}}_i$ and  $(\widehat{X}_{i,j}, j\leq  \|\widehat{{\bf K}}_i\|_1)$ the variables involved in the $i$th branching event of the spine construction for $1\leq i\leq N$. Besides, we write $E_i$ the label of the distinguished individual when the $i$th branching event occurs. Thus, if $E_i=\widehat{U}_i$, then $E_{i+1}\ne E_i$ and 
$E_{i+1}=(E_i,j)$ with $1\leq j\leq \|\widehat{{\bf K}}_i\|_1$; otherwise $E_{i+1}= E_i$.
For convenience we write 
$$A_i=\left(U_i,{\bf K}_i, (X_{i,j})_{1\leq j\leq  \|{\bf K}_i\|_1}\right), \qquad \widehat{A}_i=\left(\widehat{U}_i,\widehat{{\bf K}}_i, \,  (\widehat{X}_{i,j})_{1\leq j \leq\|\widehat{{\bf K}}_i\|_1}\right)$$
the discrete variables describing these successive branching events.\\
Let $\mathfrak A_n^{\star}$  be the subset of  non-extincted discrete trees with types in $\X$ and  $n$
internal
nodds (i.e. $n$ branching events) and initial population ${\bf x}=\{(u,x_u), \, u \in \mathfrak g\}$. Each  element of
$a\in \mathfrak A_n^{\star}$  is  a
 finite sequence  $a=(a_i)_{1\leq i \le n}$ which describes the successive branching 
 events (forgetting the time). More precisely  $a_i=(u_i, {\bf k}_i, (x_{i,j})_{ 1\leq j\leq  \|{\bf k}_i\|_1})\in \mathcal U\times \mathcal Z\times \cup_{k\geq 0} \mathcal X^k$
means that individual $u_i$ has offsprings whose types are counted by ${\bf k}_i$
and successively given by $(x_{i,j})_{ 1\leq j\leq  \|{\bf k}_i\|_1}$.
For $0\leq k\leq n$, we denote  by $\mathfrak g_k(a)\subset \mathcal U$ the labels  of individuals alive just after the $k$-th event (and before the $k+1$-th) and ${\bf z}_k(a)\in \mathcal Z$  the  vector giving the corresponding  type composition. 
We also write $y_{k}(e)$ the type of the ancestor of $e$ between these $k$-th  and $k+1$-th branching event. The fact that the tree $a\in \mathfrak A_n^{\star}$ is non-extincted means that we require that
$\mathfrak g_k(a)\ne \varnothing$ for $k\leq n$. Note also that $\mathfrak g_0(a)=\mathfrak g$.  
\begin{figure}[!t]
\begin{center}
\includegraphics[scale=0.18]{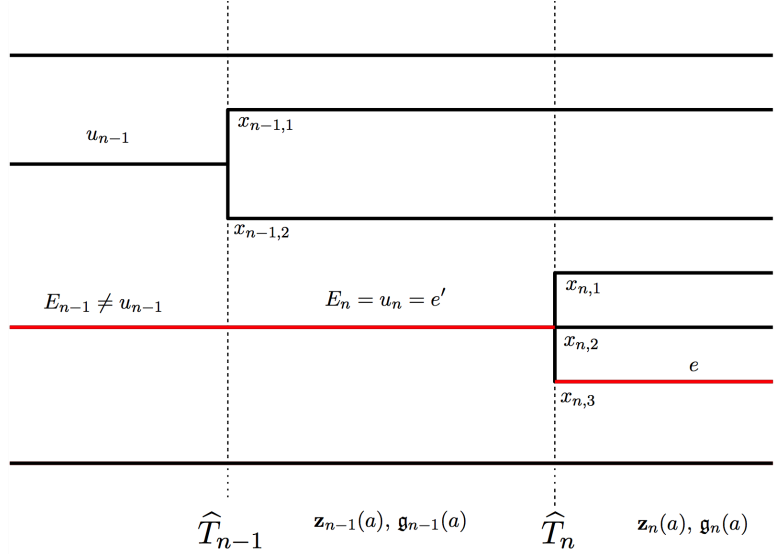}
\caption{Illustration of the construction and notation. The lines correspond to individuals and their label is indicated just above the line. Below the line we give types  for these individuals. The spine is represented in red. }
\end{center}
\end{figure}

\begin{lemma}\label{aveclesmains}
Let $n\geq 0$ and $G$ be a measurable non-negative function from $\R_+^n$.
For any  $a \in \mathfrak A_n^{\star}$  and any  $e \in \mathfrak g_{n}(a)$, 
\begin{align*}
&\E_{{\bf x}}\left(\Ind{N\geq n} G(T_1, \ldots, T_n)\, \Ind{A_i=a_i :  \,  1\leq  i \leq n}\right)\\
&\qquad =\langle {\bf v}  ,\psi(., {\bf v})\rangle\, \E_{{\bf x}}\left(\Ind{\widehat{N}\geq n, \, E_{n+1}=e} \, \mathcal W^{(a,e)}_{n}\,  G(\widehat{T}_1,\ldots,
 \widehat{T}_n) \Ind{\widehat{A}_i=a_i : \,  1\leq i \leq n}\right),
\end{align*}
where  
$$\mathcal W_n^{(a,e)}=
  \frac{\exp\left(\sum_{k=0}^{n-1}(\widehat{T}_{k+1}-\widehat{T}_k) \lambda(y_k(e),{\bf z}_k(a))\right)}{\psi(y_{n}(e),{\bf z}_n(a))} .$$
\end{lemma}
\begin{proof} The initial population ${\bf x}$ is fixed and notation is omitted in this  proof. 
For convenience, we also write ${\bf z}_n={\bf z}_n(a)$ the composition of the population between the $n$-th and $(n+1)$-th branching event and $\mathfrak{g}_n=\mathfrak{g}_n(a)$ the set of labels alive at this time.\\
We proceed by induction and start with $n=0$. 
For any $e\in \mathfrak g_0=\mathfrak g$,
$$\langle {\bf v}  ,\psi(., {\bf v})\rangle\E(\Ind{E_1=e} \, \mathcal W_{0}^{(\varnothing,e)}) =\E\left(\Ind{E_1=e} \,\frac{\langle {\bf v}  ,\psi(., {\bf v})\rangle}{\psi(x_e, {\bf v})}\right)=1.$$
Let us now consider $n\geq 1$ and assume that the identity holds for $n-1$.
We consider $G_{n}(t_i :1\leq i \leq n)=G(t_i :1\leq i \leq n-1)H(t_{n}-t_{n-1})$, where
$G$ and $H$ are measurable and non-negative and bounded respectively on $\R_+^{n-1}$ and $\R_+$.
We fix also $a=(a_i)_{1\leq i \le n}\in \mathfrak A_{n}^{\star}$ and first observe that
\begin{align}
&\E\left(\Ind{N\geq n}\, G_{n}(T_i,  \, 1\leq  i \leq n)\, \Ind{A_i=a_i :  \,  1\leq  i \leq n} \, \big\vert  \, \mathcal F_{T_{n-1}}\right) \nonumber\\
&\qquad \qquad \qquad \qquad  \qquad \qquad  =\Ind{N\geq n-1}\, G(T_i,  \, 1\leq i \leq n-1)\, \Ind{A_i=a_i :  \,  1\leq  i \leq n-1}\,
 B_n, \label{original}
\end{align}
where 
$$B_n=\E\left(\Ind{N\geq n}\, H(T_{n}-T_{n-1})\, \Ind{A_{n}=a_{n}} \,  \big\vert  \, \mathcal F_{T_{n-1}}\right)$$
and $\mathcal F_{T_{n-1}}=\sigma( T_i, A_i :  i\leq n-1)$ is the filtration generated until time $T_{n-1}$ in the original construction. 
Conditionally on $\mathcal F_{T_{n-1}}$, on the event $\{A_i=a_i :  \,  1\leq  i \leq n-1\}$ the random variable $T_{n}-T_{n-1}$ is exponentially distributed with parameter
$$\tau_n=\sum_{u \in \mathfrak{g}_{n-1} } \tau(x_u,{\bf z}_{n-1}).$$  
Considering $a_{n}$, 
 the label $u_{n}$ branches
and this latter is replaced by individual with types $(x_{n,j})_{ 1\leq j\leq  \|{\bf k}_{n}\|_1}$ and composition
${\bf k}_{n}$.  For convenience, we also write
$$Q_{n}=Q_{{\bf k}_{n}} ( x_{n,j}, 1\leq j\leq \|{\bf k}_{n}\|_1 ),$$
the probability to choose the types $( x_{n,j}, 1\leq j\leq \|{\bf k}_{n}\|_1 )$. 
On the event $\{N\geq n-1\}\cap \{A_i=a_i :  \,  1\leq  i \leq n-1\}$, we get
\begin{align}
\label{idBn}
 B_n& =\Ind{\tau_n\ne 0} \, \tau_{{\bf k}_{n}}(x_{u_{n}},{\bf z}_{n-1})\, Q_{n} \, \int_{\R_+} H(t)\, e^{-\tau_n t }  dt.
\end{align}
Similarly for $e\in \mathfrak g_{n}$, with direct ancestor $e'\in  \mathfrak g_{n-1}$ before the last branching event,
\begin{align}
 &\E\left(\Ind{\widehat{N}\geq n, \, E_{n+1}=e} \, \mathcal W_{n}^{((a_i : i \leq n),e)}\,  G_n(\widehat{T}_i,  \, 1\leq  i \leq n) \Ind{\widehat{A}_i=a_i : \,  1\leq i \leq n} \,  \big\vert  \, \widehat{ \mathcal F}_{T_{n-1}}\right) \nonumber \\
& \qquad \, =\Ind{\widehat{N}\geq n-1, \, E_{n}=e'}\, \mathcal W_{n-1}^{((a_i : i \leq n-1),e')} \,  G(\widehat{T}_i,  \, 1\leq i \leq n-1) \Ind{\widehat{A}_i=a_i : \,  1\leq i \leq n-1}\,   \frac{\psi(y',{\bf z}_{n-1})}{\psi(y,{\bf z}_{n})}  \widehat{B}_n,   \label{epine}  \end{align}
where 
$$\widehat{B}_n=\E\left(\Ind{\widehat{N}\geq n, \, E_{n+1}=e} \,  e^{(T_{n}-T_{n-1}) \lambda(y',{\bf z}_{n-1})} \,  H(\widehat{T}_{n}-\widehat{T}_{n-1}) \Ind{\hat{A}_{n}=a_{n}} \,  \big\vert  \, \mathcal F_{T_{n-1}}\right)$$
and  $y$ (resp. $y'$) is the type of the spinal individual $e$ (resp. $e'$) after (resp. before)
the $n$th branching event. We write   respectively
\Bea
\widehat{\tau}_n&=&\sum_{u \in {\mathfrak g_{n-1}}-\{e'\}}  \widehat{\tau}(x_u,y',{\bf z}_{n-1}), \qquad \widehat{\tau}_n^{ \, \star}= \widehat{\tau}^{ \, \star}(y',{\bf z}_{n-1}),
\Eea
the total branching rates of the population oustide the spine  and of the spine. Recalling that ${\bf z}_{n}={\bf z}_{n-1}-{\bf e}(x_{u_{n}})+{\bf k}_{n}$, we also write
$$
\widehat{\tau}_{n,{\bf k}_{n}}=
\tau_{{\bf k}_{n}}(x_{u_{n}},{\bf z}_{n-1}) \frac{\psi(y,{\bf z}_{n})}{ \psi(y',{\bf z}_{n-1})}
$$
the  rate at which an individual outside the spine is replaced by ${\bf k}_{n}$. If the  branching  indeed occurs outside the spine,  $y'=y$ and this rate $\widehat{\tau}_{n,{\bf k}_{n}}$ coincides with
$\widehat{\tau}_{{\bf k}_{n}}(x_{u_{n}},y',{\bf z}_{n-1})$. Besides, 
$$
 \widehat{\tau}_{n,{\bf k}_{n}}^{ \, \star}=\widehat{\tau}_{{\bf k}_{n}}^{ \, \star}(y', {\bf z}_{n-1})=\tau_{{\bf k}_{n}}(y',{\bf z}_{n-1}) \frac{\langle {\bf k}_{n},\psi(.,{\bf z}_{n})\rangle}{ \psi(y',{\bf z}_{n-1})}
$$
yields the branching rates for the spine.  If the  branching event indeed concerns 
 the spine,  $y'=x_{u_n}$ may differ from $y$. 
Similarly, the probability to choose a spine with type 
$y$  is
$$q_{n}=q_{y}({\bf k}_{n},{\bf z}_{n})=\frac{ \psi(y,{\bf z}_{n})}{\langle{\bf k}_{n},\psi(.,{\bf z}_{n})\rangle}.$$  
We distinguish  two cases, corresponding to the fact that the $n$th branching event
concerns  the spine or not, i.e. either
$u_{n}=e'$  or ($u_n \ne e'$ and $y=y'$).
On the event $\{\widehat{N}\geq n-1, \, E_{n-1}=e'\}\cap \{\widehat{A}_i=a_i : \,  1\leq i \leq n-1\}$, the time $\widehat{T}_n-\widehat{T}_{n-1}$
is exponentially distributed with parameter
$\widehat{\tau}_n+\widehat{\tau}_n^{\, \star}$ and we get
\begin{align*}
\widehat{B}_n
& = \Ind{\widehat{\tau}_n+\widehat{\tau}_n^{\, \star}\ne 0, \, u_{n}\ne e'} 
\int_{\R_+} H(t) e^{t ( \lambda(y',{\bf z}_{n-1})-(\widehat{\tau}_n+\widehat{\tau}_n^{ \, \star}))}\, \widehat{\tau}_{n,{\bf k}_{n}}\, Q_{n}\, dt \\
& \quad +\Ind{\widehat{\tau}_n+\widehat{\tau}_n^{ \, \star}\ne 0, \, u_{n}= e'} \int_{\R_+} H(t) e^{t (\lambda(y',{\bf z}_{n-1})-(\widehat{\tau}_n+\widehat{\tau}_n^{ \, \star}))} \widehat{\tau}_{n,{\bf k}_{n}}^{ \, \star}\,  q_{n}Q_{n}\, dt \\
& =\Ind{\widehat{\tau}_n+\widehat{\tau}_n^{ \, \star}\ne 0 } \,  \frac{\psi(y,{\bf z}_{n})}{\psi(y',{\bf z}_{n-1})} \, \tau_{{\bf k}_{n}}(x_{u_{n}},{\bf z}_{n-1}) \, Q_{n} \,  \int_{\R_+} H(t) e^{t (\lambda(y',{\bf z}_{n-1})-(\widehat{\tau}_n+\widehat{\tau}_n^{ \, \star}))}\,   dt.
\end{align*}
 Adding that
 by definition of $\lambda$,
$$\widehat{\tau}_n+\widehat{\tau}_n^{ \, \star}-\lambda(y',{\bf z}_{n-1}) =\tau_n,$$
we obtain  from \eqref{epine}
\begin{align*}
 &\E\left(\Ind{\widehat{N}\geq n, \, E_{n+1}=e} \, \mathcal W_{n}^{((a_i : i \leq n),e)}\,  G_n(\widehat{T}_i,  \, 1\leq  i \leq n) \Ind{\widehat{A}_i=a_i : \,  1\leq i \leq n} \,  \big\vert  \, \widehat{ \mathcal F}_{T_{n-1}}\right)  \nonumber \\
& \qquad \, =\Ind{\widehat{N}\geq n-1, \, E_{n}=e'}\, \mathcal W_{n-1}^{((a_i : i \leq n-1),e')} \,  G(\widehat{T}_i,  \, 1\leq i \leq n-1) \Ind{\widehat{A}_i=a_i : \,  1\leq i \leq n-1, \, \widehat{\tau}_n+\widehat{\tau}_n^{\star}\ne 0 }\, \\
&\qquad \qquad \qquad   \times \int_{\R_+} H(t) e^{-\tau_n t}\,   dt \, \times  \tau_{{\bf k}_{n}}(x_{u_{n}},{\bf z}_{n-1}) \, Q_{n}.
\end{align*}
\cmv{A ecrire differemment?}
Using \eqref{original} and \eqref{idBn} and the fact that $\widehat{\tau}_n+\widehat{\tau}_n^{\star}= 0$ is equivalent to $\tau_n=0$, 
the induction hypothesis  ensures
\begin{align*}
&\E\left(\Ind{N\geq n}\, G_{n}(T_i,  \,1\leq i \leq n)\, \Ind{A_i=a_i :  \,  1\leq  i \leq n} \right)\\
&\qquad =\langle {\bf v}  ,\psi(., {\bf v})\rangle\, \E\left(\Ind{\widehat{N}\geq n, \, E_{n+1}=e} \, \mathcal W^{((a_i : i \leq n),e)}_{n}\,  G_{n}(\widehat{T}_i,  \, 1\leq i \leq n) \Ind{\widehat{A}_i=a_i : \,  1\leq i \leq n}\right)
\end{align*}
by conditioning both sides with respect to their filtration until the $n+1$th branching event. It ends the proof by a monotone class argument.
\end{proof}

\begin{proof}[Proof of Theorem \ref{main}]
The result is a consequence of the previous lemma. On the event $\{N<\infty\}$, 
we  set $T_{n}=+\infty$ for   $n >N$.
For  each $t\geq 0$ and $n\geq 0$ and $e\in \mathcal U$, we introduce
 a measurable non-negative function $G^{t,e}_n$  from $\R_+^n\times \mathfrak A_n^{\star}$ such that, on the event $\{T_n\leq t < T_{n+1}, N\geq n\}$ we have 
$$F(\T(t),e)p_e(\T(t))= G^{t,e}_n(T_1, \ldots, T_n, A_1,\ldots,A_n) \quad \text{a.s}.$$
  Then
\begin{align*}
&\E_{{\bf x}}\left( \Ind{T_{\text{Exp}}>t, \, \mathbb G(t)\ne \varnothing} \, F(\T(t),U(t))\right)\\
&\qquad   =\sum_{\substack{n\geq 0,\\  a\in \mathfrak A_n^{\star},\, e\in \mathfrak g_n(a)}}
\E_{{\bf x}}\left(F(\T(t),e)p_e(\T(t)){\bf 1}_{A_i=a_i : 1\leq i\leq n,   \, T_n\leq t < T_{n+1}, N\geq n} \right)\\
&\qquad =\sum_{\substack{n\geq 0,\\  a\in \mathfrak A_n^{\star},\, e\in \mathfrak g_n(a)}}
F^{t,e}_n(a), 
\end{align*}
where
 \begin{align*}
F^{t,e}_n(a)= &\E_{{\bf x}}\left(  G^{t,e}_n(T_1, \ldots, T_n, a_1,\ldots,a_n)f_t(T_n,  a_1,\ldots,a_n) \Ind{A_i=a_i : 1\leq i\leq n, T_n\leq t , N\geq n}\right)
 \end{align*}
and 
$$ f_t(T_n,  a_1,\ldots,a_n)=\P(T_{n+1}>t \vert T_n, A_n=a_n, \ldots, A_1=a_1).$$
To end the proof, we apply Lemma \ref{aveclesmains} to express $F^{t,e}_n(a)$ in terms of the spine construction
 and writting $y_n(e)$ the type of the spine at the $n$th event, we use that 
 \begin{align*}
& \Ind{ E_{n+1}=e}f_t(\widehat{T}_n,  a_1,\ldots,a_n)\\
&\qquad =\Ind{ E_{n+1}=e}e^{(t-\widehat{T}_n) \lambda(y_n(e),{\bf z}_n(a))} \P(\widehat{T}_{n+1}>t\vert \widehat{T}_n, \widehat{A}_n=a_n, \ldots, \widehat{A}_1=a_1).
\end{align*}
This latter identity is proved  following the last lines  as in the proof of Lemma \ref{aveclesmains}. 
\end{proof}
\subsection{Positive semigroup and   martingale}
For each $(r,{\bf v})\in \overline{\mathcal Z}$, we associate an initial labeling
${\bf x}=x({\bf v})=((u,x_u) : u \in \mathfrak g)$, where $x_u$ is the type of $u\in  \mathfrak g$ and  $\#\{ u \in  \mathfrak g : x_u=x\}={\bf v}_x$. Since ${\bf v}_r\geq 1$, we can also  associate a label $u_r \in \mathfrak g$ such that $x_{u_r}=r$ .
For any $t\geq 0$  and $f$   function from $\overline{\mathcal Z}$ to $\R_+\cup\{+\infty\}$, we define for any $(r,{\bf v})\in \overline{\mathcal Z}$,
$$M_tf(r,{\bf v})=\E_{x({\bf v})}\left( \Ind{T_{\text{Exp}}>t}\, \sum_{  u\in \G(t), \, u\succcurlyeq u_r} f(Z_u(t), {\bf Z}(t))\right),$$
where  $Z$ and ${\bf Z}$ are   defined in Section \ref{model} with  initial condition $x({\bf v})$. \cmv{Faire le lien entre $M$ et $\mathcal G$ : c'est le generateur, modulo l'explosion}
This corresponds to the first moment associated to the empirical measure of the descendance of a specific initial individual, together with the total composition of the population. We   observe that  this definition  
 does not depend on the choice of the labels of $x({\bf v})$ and  $u_r$. We can also write \\
$$M_tf(r,{\bf v})=\E_{x({\bf v})}\left( \Ind{T_{\text{Exp}}>t}\, \langle {\bf Z}^{(u_r)}(t), f( ., {\bf Z}(t))\rangle\right),$$
where ${\bf Z}_x^{(u)}(t)=\#\{ v \in \G(t) \, : \, v\succcurlyeq u, \, Z_v(t)=x \}$ is the number of individuals with type $x$ at time $t$
 who are descendant of $u$.\\

Recall that $\psi\in \mathfrak D_{\mathcal G}$ is  positive   on  $\overline{\mathcal Z}$ and  satisfies $\eqref{condsum}$ and   that $\lambda=\mathcal G \psi/\psi$
on  $\overline{\mathcal Z}.$
Recall also that   ${\bf \Xi}$ is the  process counting types in the  $\psi$-spine
construction    and  $Y(t)=\Xi_{E(t)}(t)$ is the type of the spine at time $t$.
Observe that $(Y,{\bf \Xi})$ is a jump Markov process whose jump rates are determined by  $\widehat{\tau}_{{\bf k}}$ and $\widehat{\tau}^{\, \star}_{{\bf k}}$ for ${\bf k}\in \mathcal Z$. It starts from $(Y(0),{\bf \Xi}(0))=(r,{\bf v})$. For $u\in \G(t)$, we write $Z_u(s)$ the type of the (unique) ancestor of $u$ at time $s\leq t$.

\begin{proposition} \label{sg} 
 $(M_t)_{t\geq 0}$ is a positive semigroup  on the set of  functions from $\overline{\mathcal Z}$ to $\R_+\cup\{+\infty\}$.  Besides,  for any $t\geq 0,$ for any non-negative  function 
 $f$ on $\overline{\mathcal Z}$ and $(r,{\bf v})\in \overline{\mathcal Z}$,
$$M_tf(r,{\bf v})=\psi(r,{\bf v}) \, \E_{(r,{\bf v})}\left( \Ind{\widehat{T}_{\emph{Exp}}>t}\, \frac{e^{\int_0^t \lambda (Y(s),{\bf \Xi}(s)) ds}}{\psi(Y(t),{\bf \Xi}(t))} f(Y(t),{\bf \Xi}(t))\right).$$
Furthermore, for any $G$ measurable function from $\mathbb D([0,t], \X\times \mathcal  Z)
$ to $\R_+$,
\begin{align*}
&\E_{x({\bf v})}\left(\Ind{T_{\emph{Exp}}>t}\,\sum_{ u\in \G(t)} \psi(Z_u(t),{\bf Z}(t))G((Z_u(s),{\bf Z}(s))_{s\leq t})\right) \,
\nonumber \\
&\qquad \qquad =\langle   {\bf v},\psi(., {\bf v}) \rangle \E_{x({\bf v})}\left( \Ind{\widehat{T}_{\emph{Exp}}>t}\, e^{\int_0^t \lambda(Y(s),{\bf \Xi}(s))ds}\, G((Y(s),{\bf \Xi} (s))_{s\leq t})\right).  
\end{align*}
\end{proposition}
We first observe that the generator of the semigroup $M$ is the linear operator 
$\mathcal G$ introduced above. This will be made explicit in applications.
This proposition  provides   a Feynman Kac representation  of the semigroup and  a so-called many-to-one formula for the population.
 We refer to \cite{DelMoral} for a general reference  on Feynman Kac formulae
 and Biggins and Kyprianou \cite{BK} for related works on  multiplicative martingales. For such representations in the context of structured branching processes and in particular for fragmentations or growth fragmentations, we mention the works of Bertoin   \cite{Bertoinbook,  Bertoin2} and Cloez \cite{Cloez} and Marguet \cite{Marguet}. We note that the event $\{\widehat{T}_{\text{Exp}}>t\}$ is measurable with respect to filtration associated to the process ${\bf \Xi}$ since this event 
is characterized by the absence of 
 accumulation of jumps for ${\bf \Xi}$ before time $t$. \\

\begin{proof}  We omit initial condition  in notation.
To prove that $M$ is a semigroup,  we  condition by the filtration $\mathcal F_t$  generated by the original process until  time $t$. For any 
$u\in \mathcal U$ and non-negative function $f$,
\begin{align*}
\E\left( \Ind{T_{\text{Exp}}>t+s, \, u\in \G(t)} \langle {\bf Z}^{(u)}(t+s), f( ., {\bf Z}(t+s))\rangle \, \big \vert \, \mathcal F_t\right)
&= \Ind{T_{\text{Exp}}>t, \, u\in \G(t)} M_sf(Z_u(t),{\bf Z}(t)).
\end{align*}
We get
\Bea
&&M_{t+s}f(r,{\bf v})\\
&& \qquad = \E\Big( \Ind{T_{\text{Exp}}>t}\, \sum_{  u\in \G(t), \, u\succcurlyeq u_r} 
\E\left( \Ind{T_{\text{Exp}}>t+s, \, u\in \G(t)} \langle {\bf Z}^{(u)}(t+s), f( ., {\bf Z}(t+s))\rangle \, \big \vert \, \mathcal F_t\right)\Big)\\
&&\qquad =
\E\left(  \Ind{T_{\text{Exp}}>t}
\langle {\bf Z}^{(u_r)}(t), M_sf( ., {\bf Z}(t))\rangle
\right)=M_t(M_sf)(r,{\bf v}).
\Eea
To prove the Feynmac Kac representation of the semigroup $M$ and get the ancestral lineage of a typical individual, 
we prove 
that  for $t\geq 0$,
\bea
&&\E\Big(  \Ind{T_{\text{Exp}}>t} \, \sum_{  u\in \G(t), \, u\succcurlyeq u_r} \psi(Z_u(t),{\bf Z}(t))G((Z_u(s),{\bf Z}(s))_{s\leq t})\Big)\nonumber\\
&&\qquad \qquad  \quad =\psi(r, {\bf v}) \E_{{(r,\bf v)}}\left(  \Ind{\widehat{T}_{\text{Exp}}>t} \,  e^{\int_0^t \lambda(Y(s),{\bf \Xi}(s))ds}\, G((Y(s),{\bf \Xi} (s))_{s\leq t})\right). \label{iddd}
\eea
Indeed, we can apply Theorem \ref{main} to
\Bea
F(\mathfrak{t},u)&= &\#\{ v\in  {\mathfrak g}(t) : \, v\succcurlyeq u_r\}\, \psi(z_u(t),{\bf z}(t))\, G((z_u(s),{\bf z}(s))_{s\leq t}), \\
 p_u(\mathfrak{t}) &=&\frac{{\bf 1}_{u\in {\mathfrak g}(t),  \, u\succcurlyeq u_r}}{\#\{ v\in  {\mathfrak g}(t) : \, v\succcurlyeq u_r\}},
\Eea
where ${\mathfrak g}(t)$ is the set of labels of $\mathfrak{t}$ alive at time $t$, ${\bf z}(t)$ is the type composition at time $t$ of the population 
and $z_u(t)$ the type of individual $u$ at time $t$.
We   observe  that $\E\left(\mathbf 1_{T_{\text{Exp}}>t, \, \G(t)\ne \varnothing} \, F(\T(t),U(t))\right)$ gives the left hand side of 
\eqref{iddd}
by exploiting the law  $p(\mathcal T(t))$ of $U(t)$ conditionally on $\mathcal T(t)$, while 
$\E\left(\mathcal W(t)\, F(\A(t), E(t))\right)$  yields the right hand side of  \eqref{iddd} by  conditioning by $E(0)=u_r$. We can also remark that  \eqref{iddd}
 amounts to  a spine construction with  initial condition $E(0)=e$, $Y(0)=r$, which 
focuses on the  lineages of individuals whose initial ancestor is $u_r$. This would provide an alternative proof.
Identity \eqref{iddd} proves the first expected expression of semigroup $M$
by considering marginal functions at time $t$. It also yields the second 
one    by summation over initial individuals, which ends the proof.
\end{proof} 

\cmv{Supermartingale utile si explosion ?}
\begin{proposition} \label{martin}
If $T_{\emph{Exp}}=+\infty$ p.s. and $\widehat{T}_{\emph{Exp}}=+\infty$ p.s., then
$$M(t)= \sum_{ u\in \G(t)}
e^{-\int_0^t  \lambda(Z_u(s),{\bf Z}(s))\, ds}\,\psi(Z_u(t), {\bf Z}(t))$$
is a non-negative martingale with respect to the filtration $(\mathcal F_t)_{t\geq 0}$ generated by the original process $Z$.  Furthermore,  it converges a.s. to $W\in [0,\infty)$.
\end{proposition}
The proof of the martingale property uses
$\widehat{T}_{\text{Exp}}=+\infty$. Indeed, otherwise mass decays and in the case of jump process, we do not have a direct compensation via a killing rate. Observe from  \eqref{manytoone}below  that under the condition $T_{\text{Exp}}=+\infty$ p.s, the fact that $M$ is a martingale (and not only a local martingale) is equivalent to $\widehat{T}_{\text{Exp}}=+\infty$ a.s.\\
Besides, 
the limit $W$ may degenerate to $0$. In the case of branching processes,
the criterion  for non-degenerescence 
is the $L\log L$ condition for reproduction law, coming from Kesten and Stigum theorem.
In Section \ref{LlogL}, we deal with a counterpart with interactions in the single type case, following the spinal approach of \cite{LPP}  for Galton Watson processes. It corresponds to the case $\psi=1$. Then $\lambda (x,{\bf z})= \sum_{{\bf k}\in \mathcal Z} \tau_{{\bf k}}(x,{\bf z}) \| {\bf k} \|_1 \, -\tau(x,{\bf z})$
yields   the growth rate induced by individuals of  type $x$ among population
${\bf z}$.\\
\cmv{Mettre des $h$ quand c'est vect propre ?}
Finally, we stress that the case when the semigroup $M$ has a positive eigenfunction (harmonic function)
allows to simplify the exponential term, since $\lambda$ is then  constant.
It is of particular interest and will be exploited in applications of the next sections. In population dynamics and genetics, this eigenfunction corresponds to the notion of \emph{reproductive value}, giving the long term impact in terms of population size of a given individual (depending on its trait).
We refer to \cite{MS, BCGM} and references therein
for general results ensuring existence and/or uniqueness of eigenelements of positive semigroup in related contexts.

\begin{proof} The initial condition ${\bf x}$ is fixed and omitted in notation.
The fact $T_{\text{Exp}}=+\infty$ p.s and Proposition \ref{sg} applied 
to $$G((Z_u(s),{\bf Z}(s))_{s\leq t})=e^{-\int_0^t \lambda(Z_u(s),{\bf Z}(s))ds} \psi(Z_u(t),{\bf Z}(t))$$
ensure that 
\begin{align}
\langle   {\bf v},\psi(., {\bf v}) \rangle \P\left(\widehat{T}_{\text{Exp}}>t\right)=&\E\left(\sum_{ u\in \G(t)} e^{-\int_0^t \lambda(Z_u(s),{\bf Z}(s))ds} \psi(Z_u(t),{\bf Z}(t))\right)  \label{manytoone}
\end{align}
for any $t\geq 0$.  This identity  guarantees the integrability of $M$.
Similarly Markov property and $\eqref{iddd}$ allow to write for $u,t$ fixed, on the event $u\in \G(t)$,
\begin{align*}
& \E\left(  \sum_{  v \in \G({t+s}),  \, v \succcurlyeq u} e^{-\int_t^{t+s} \lambda (Z_u(\tau),{\bf Z}(\tau)) d\tau}\psi(Z_v(t+s),{\bf Z}(t+s))  \, \bigg\vert \, \mathcal F_t\right) =   \psi(Z_u(t),{\bf Z}(t)) 
\end{align*}
since $\widehat{T}_{\text{Exp}}=\infty$ a.s. We get
$$
\E(M(t+s) \vert \mathcal F_t)=  \sum_{ u\in \G(t)} e^{-\int_0^t \lambda (Z_u(\tau),{\bf Z}(\tau))d\tau} \psi(Z_u(t),{\bf Z}(t))=M(t),
$$
which proves the proposition.
\end{proof}

\section{Single type density dependent Markov process  and neutral evolution} 
\label{dim1}
In this section, we consider   single type populations and  some  issues which have originally motivated this work. 
  In that case, when the size of the population is $z\in \N$, each individual branches and is replaced by $k$ individuals with rate $\tau_k(z)$, for 
$k\in\N_0$. We forget bold letters in this section for single type case and 
$(Z(t))_{t\geq 0}$ is the jump Markov process on $\mathbb N_0$ giving the population size along time. \\
We consider $\psi :\N\rightarrow (0,\infty)$ and 
the $\psi$-spine construction is as follows. The distinguished individual    is replaced by  $k\in \N$ individuals at rate 
$$\widehat{\tau}_{k}^{\, \star}(z)=k\tau_{k}( z) \frac{\psi(z-1 + k)}{ \psi(z)} \qquad (z\geq 1).$$
Among these offsprings, each individual is chosen to be the new label of the spine with probability $1/k$.\\
The individuals but  the spine
 branch  and are replaced by $k\in \N_0$ individuals at rate 
$$\widehat{\tau}_{k}(z)=\tau_{ k}(z) \frac{\psi(z-1+ k)}{ \psi(z)} \qquad (z\geq 2).$$
We observe that the size $\Xi$  of the population in the $\psi$-spine construction
is a density dependent Markov process with transition rates from $z$ to $z+k-1$ equal to 
$$(k+z-1)\tau_{k}( z) \frac{\psi(z-1 + k)}{ \psi(z)}  \qquad (z\geq 1).$$
Thus, $\Xi$ is a population process with individual  branching rates $\tau_{k}(z) \psi(z+k)/\psi(z)$, plus additional size depend immigration, where  $k\geq 1$ immigrants arrive in the population of size $z\geq 1$ at rate $k\, \tau_k(z)\, \psi(z-1+k)/\psi(z).$\\
Generator $\mathcal G$ is now defined for real valued functions $f$ on 
$\mathbb N$ and writes  for $z\geq 1$ as
\begin{align*}
\mathcal Gf(z)
&=\sum_{k\in \N_0} \tau_{k} (z) (z+k-1)f(z+k-1)-z\tau(z) f(z).
\end{align*}
Function $\lambda=\mathcal G\psi/\psi$ becomes for $z\geq 1$,
$$ \lambda(z)=\sum_{k\in \N_0} \tau_{k} (z) (z+k-1)\frac{\psi(z+k-1)}{\psi(z)}-z\tau(z).$$
\subsection{Harmonic function}
Exchangeability in the single type case suggests the choice 
$\psi(z)=1/z$ for $z\geq 1$.
We get $\lambda(z)=0$ if $z\geq 2$ and $\lambda(1)=-\tau_0(1)$. In particular the inverse function is an eigenelement associated with the eigenvalue $\lambda=0$ when the process cannot reach (and be absorbed) in $0$, i.e. in the case
$\tau_0(1)=0$. 
Let us then consider    a uniform choice $U(t)$ among individuals  alive at time $t$. Conditionally on $\mathbb G(t)$, we assume that this variable is independent of $\mathcal T(t)$ and uniformly distributed $\mathbb G(t)$, when this latter is non empty.
Since here $\lambda=0$, we get $\mathcal W(t)={\bf 1}_{T_{\text{Exp}}>t}$ a.s. and  Theorem \ref{main} yields:
\begin{proposition} \label{und} Assume
$\tau_0(1)=0$.   
Then, for any $t\geq 0$, ${\bf 1}_{T_{\emph{Exp}}>t}(\T(t), U(t))$
is distributed as ${\bf 1}_{\widehat{T}_{\emph{Exp}}>t}(\A(t),E(t))$, where $(\A,E)$ is the  $1/z$-spine construction. 
\end{proposition}
This $1/z$-spine construction consists in a single type density dependent Markov process with a distinguished individual and individual jump rates 
$$\widehat{\tau}_{k}^{ \, \star}(z)=k\, \tau_{k}( z)  \frac{z}{z-1 + k}, 
\qquad \widehat{\tau}_{k}(z)=\tau_{ k}(z) \, \frac{z}{z-1 + k}$$
for $z\geq 1,k\geq 0$. We  recover the fact that the process $\Xi$ counting the size of the population in the $1/z$-spine construction is distributed as the original process $Z$.
We give a  consequence about ancestral lineage of samples, which will be useful. We consider the case when the size of the population  of the spine construction $\# \V(t)$
converges in law   to a stationary distribution  $\pi=(\pi_z)_{z\geq 1}$. 
Then,  the number of branching events with $k$ offsprings along the ancestral lineage of a uniform sample in $\mathbb G(t)$ grows linearly with rate 
$$\widehat{\pi}_k=k\, \sum_{z\geq 1} \pi_z \, \tau_{k}( z) \, \frac{z}{z-1 + k}.$$
Before applying this result, we mention that in the case when $\tau_0(1)\ne 0$,  an analogous  result can be stated conditionally 
on the survival of the process. The eigenfunction $\psi$ is then non-explicit in general, but can be written as $h(z)/z$ for $z\geq 1$, where $h$ is the harmonic function of the killed process. It allows in particular to describe sampling in the quasistionnary regime, i.e. when the process conditioned to survive at a given time converges in law. 
In that case the  process $\Xi$ survives a.s. but the original process dies out. \\

{\bf A growth fragmentation model with competition.}
We consider a neutral model of dividing cells including competition, which induces death of cells. The mass of the cell grows during its life at a fixed exponential speed, and two mechanisms may regulate this mass : division (random splitting of the mass)
and death  (with individual death rate of cells increasing with total number of cells).
Without interactions, for branching structures, such processes  have received lots of attention, including deterministic, random and structured frameworks. We refer  e.g. to
 \cite{BCGM, BW,  Bertoin2, Cloez, Marguet} and references therein. We assume for this example that cells divide in two daughter cells
\begin{align}
\label{naissmort}
b_z=\tau_2(z), \quad d_z=\tau_0(z), \quad \tau_1(z)= \tau_k(z)=0 \quad \text{ for  }z\geq 1, k\geq 3.
\end{align} We  assume also that the individual birth rate is bounded
and death is only caused by competition  : 
\cmv{Ref non explosion}
\begin{equation}
\label{borneb}
\sup_{z\geq 1} b_z \leq \overline{b},  \qquad  d_1=0.
\end{equation}
So  process $Z$ is well defined and positive for any time and we are exploiting the $1/z$ spine construction.  
Each cell is now equipped with a size, which grows at exponential rate $r>0$.
  Let us denote by $(\zeta_u(t))_{u \in \G(t)}$ the process giving the size of each
  cell alive at time $t$. Thus, between two jumps of the cell population, 
 $$\zeta_u'(t)=r\zeta_u(t).$$
When the cell dies, its mass is lost. When it divides, it is shared randomly between each daughter cell, using a random fraction $F\in (0,1)$ a.s. More precisely, we draw an i.i.d.  family of r.v. $(F_u)_{u\in \mathcal U}$ 
distributed as $F$ and when cell $u$ divides at time $t$ with mass $\zeta_{u}(t-)$, its two offsprings
get masses 
$$(\zeta_{(u,1)}(t),\zeta_{(u,2)}(t))=(F_u\zeta_{u}(t-),(1-F_u)\zeta_{u}(t-)).$$ 
Without loss of generality, we assume  that 
$F$ is distributed as $1-F$, i.e.  the law of $F$ is symmetric  with respect to one half.
Let us refer to \cite{BDMT, Marguet} for similar constructions in general context of branching processes. 
We start for simplicity from a single cell with size $\zeta_0>0$ : $Z(0)=1$, $\zeta_1(0)=\zeta_0$. Let us give a trajectorial description of the population process together with the spine individual. For that purpose and convenice, we use a Poisson representation 
for constructing the original birth and death process $Z$, given by a Poisson point measure and use the same measure for the spine construction.
\cmv{est ce ambigu d'utiliser $Z$ a chaque fois sachant que pour l'epine c'est $\Xi$ mais c'est pareil ici? ajouter IW en ref}
More precisely, we define
 the  process
$(Z(t),\zeta^\star(t))_{t\geq 0}$
 as the unique strong solution of the following stochastic differential equation
\begin{align*}
Z(t)&=1+\int_0^t\int_{\R_+^2} \left({\bf 1}_{u\leq  Z(s-)b_{Z(s-)}} -{\bf 1}_{ Z(s-) b_{Z(s-)}<u\leq Z(s-)(b_{Z(s-)}+d_{Z(s-)})}\right)\mathcal N(ds,du,df),\\
\zeta^\star(t)&=\zeta_0+\int_0^t r \zeta^\star(s)\, ds-\int_0^t \int_{\R_+\times (0,1)} (1-f) \zeta^\star(s) {\bf 1}_{u\leq  2 b_{Z(s-)} Z(s-)/(Z(s-)+1)} \, \mathcal N(ds,du,df),
\end{align*}
where $\mathcal N$ is a Poisson point measure on $\R_+^2\times (0,1)$
with intensity $dsdu\P(F\in df)$. 
Existence and strong uniqueness are classical
and we refer e.g. to \cite{IW,BM}.
In words, $\zeta^\star$ is a Markov process growing exponentially at speed $r$, which undergoes multiplicative jumps distributed as $F$. These jumps occur at the  birth rate along the spine, which itself lives under population $Z$. Since explosion is here excluded by assumption, Proposition \ref{und} yields:
\begin{proposition} \label{neutre} Assume that \eqref{naissmort} and \eqref{borneb} hold. Let  $t\geq 0$ and $U(t)$ be a uniform choice among $\mathbb G(t)$, which is independent of $(\zeta_u(s))_{s\geq 0, u \in \G(s)}$ conditionally on $\mathbb G(t)$. Then,
$(Z(t),\zeta_{U(t)}(t))$ is distributed as $(Z(t),\zeta^\star(t))$.
\end{proposition} 
We stress that this identity in law holds (only) for fixed time $t$, not for the full processes.
\cmv{A mettre apres juste a l'enonce ?}
This result allows to use Markov techniques to study the regulation of the size of cells through a typical (uniformly chosen) lineage. In particular, we can state here a new transition phase exploiting Birkhoff ergodic theorem, when the number of cells is
regulated by competition.
Thus  we also  assume now
that the Markov process $Z$ is irreducible and positive recurrent on 
$\mathbb N$ (see \cite{KM, Karlin} for  explicit conditions). 
Then  $Z(t)$ converges in law to the unique    stationary distribution $\pi=(\pi_z)_{z\geq 1}$ as $t$ tends to infinity and we set  $$\widehat{\pi}= 2 \sum_{z\geq 1} \pi_z b_z \frac{z}{z+1}.$$
Finally, we make the following moment assumption:
\begin{equation}
\label{momentF}
\E(\log (F)^2) <\infty
\end{equation} 
and  we get under these conditions the following classification for this model.
\cmv{DONNNER CONDITiON SUFFISANTE IRREductibilite recurrence poisitive, Dire que c'est nouveau (mm si on dit pas que c'est pas obtenable autrement souligner aussi au dessus l'interet d'vaoir pas juste le regim  asymptotique avec des vitesses, effets de la cond intiale etc }
\begin{corollary}  \label{coronord}  Assume that \eqref{naissmort}-\eqref{borneb}-\eqref{momentF} hold and that $Z$ is irreducible and positive recurrent. \\
i) If  $r <\E(\log(1/F))\, \widehat{\pi},$
then $\zeta^\star(t)$ tends a.s. to $0$ as $t\rightarrow\infty$ and 
$$\lim_{t\rightarrow \infty} \max \{  \zeta_u(t)\, : \, u \in \G(t) \}=0 \quad \text{in probability}.$$
ii) If  $r >\E(\log(1/F)) \, \widehat{\pi},$ then $\zeta^\star(t)$ tends a.s. to infinity as $t\rightarrow\infty$ and
$$\lim_{t\rightarrow \infty} \min \{  \zeta_u(t)\, : \, u \in \G(t) \}= +\infty \quad \text{in probability}.$$
\end{corollary}
\begin{proof}  Recalling the SDE representation of $\zeta^\star$ given above, we have 
$$\log(\zeta^\star(t))=\log(\zeta_0)+rt+\int_0^t  \int_{\R_+\times (0,1)} \log(f) {\bf 1}_{u\leq  2 b_{Z(s-)} Z(s-)/(Z(s-)+1)} \, \mathcal N(ds,du,df).$$
The classification and asymptotic behavior of $\zeta^\star$   is then inherited from ergodic averaging of Birkhoff theorem. Indeed, writing 
$h(z)=2b_z z/(z+1) $ and $\alpha=\E(\log(F))$,
\begin{align*}
\log(\zeta^\star(t))&=\log(\zeta_0)
+rt+\alpha \int_0^t    h(Z(s))ds + M(t),
\end{align*}
where 
$$M(t)= \int_0^t  \int_{\R_+\times (0,1)} \log(f) {\bf 1}_{u\leq  h(Z(s-))} \, \widetilde{\mathcal N}(ds,du,df)$$
and 
$\widetilde{\mathcal N}$ is the compensated measure of $\mathcal N$. Birkhoff theorem for continuous time Markov processes \cite{Norris} ensures that
$$\frac{1}{t} \int_0^t    h(Z(s))ds\stackrel{t\rightarrow \infty}{\longrightarrow} \sum_{z\geq 1}
h(z)\pi_z \qquad \text{a.s.}$$
since $h$ is bounded by Assumption \eqref{borneb}. Besides
$(M(t))_{t\geq 0}$ 
is a martingale with  bounded quadratic variation on finite time intervals by Assumption \eqref{momentF}.
We deduce that 
$$\frac{1}{t} M(t) \stackrel{t\rightarrow \infty}{\longrightarrow} 0 \qquad \text{a.s.}$$
and we can conclude that $\log(\zeta^\star(t))$ tends to $+\infty$ or $-\infty$ depending
on the fact that $r+\alpha \sum_{k\geq 0}
h(z)\pi_z $ is positive or negative.\\
$\newline$

\noindent We conclude on the original process by using  Proposition \ref{neutre}. Indeed, 
let $\varepsilon>0$ and $A>0$,
$$\P\left(\max \{  \zeta_u(t)\, : \, u \in \G(t) \}\geq\varepsilon,\, \# \G(t)\leq A\right)\leq A\P\left( \zeta_{U(t)}(t)\geq\varepsilon\right)=A\P\left( \zeta^{ \, \star}(t)\geq\varepsilon\right)$$   and the right hand tends to $0$  if   $r <\E(\log(1/F)) \widehat{\pi}$. We conclude for  $i)$ by letting $A$ go to infinity and $\varepsilon$ go to $0$ and by using that $\# \G(t)=Z(t)$ is stochastically bounded. The  other case is treated similarly.
\end{proof}
Let us comment briefly this result and the proof. The assumptions of boundedness of the individual birth rate $b_z$ and the second moment of $\log F$ could be probably relaxed using finer ergodic techniques. 
The critical case is interesting. We expect that in general $\zeta^\star$ oscillates a.s. and that for any $\varepsilon>0$,
 $$\limsup_{t\rightarrow \infty} \P(\max \{  \zeta_u(t)\, : \, u \in \G(t) \} \leq \varepsilon)=1, \quad \limsup_{t\rightarrow \infty} \P( \min \{  \zeta_u(t)\, : \, u \in \G(t) \} \geq 1/\varepsilon)=1.$$

We illustrate now Corollary \ref{coronord}  with a classical logistic competition model and  the criterion becomes explicit. The individual birth rate is fixed and equals to $b>0$ and the competition coefficient with other cells is $c>0$:
$$b_z=b, \qquad d_z=c(z-1) \qquad (z\geq 1).$$
The stationary probability $\pi$ of the population size  is  
$$
\pi_z
=\frac{1}{e^{b/c}-1}\Big(\frac{b}{c}\Big)^z\frac{1}{z!} \qquad (z\geq 1).
$$
The criterion for the regulation of the growth of mass can be given in terms of the parameters $b$ (birth) and $c$  (competition) and $r$ (growth) and $F$ (random repartition at division):
$$r<2   b\left(1- \frac{c}{b}+\frac{1}{e^{b/c}-1}\right) \, \E(\log(1/F)).$$
\cmv{c'est pas exactement une carrying capacity}
Letting  $c$ tend to $0$ allows to recover the expected criterion for  branching process,
with  classical accelerated  rate of branching $2b$ along a typical lineage, see e.g. \cite{BM}.   Both division (by splitting) and competition (by killing) participate to the regulation of the growth of the cell mass. The threshold above (so as 
the mass growth rate, see the proof) makes appear 
 the function
$f(b,b/c)$, where $f(b,y)=b(1-1/y+1/(e^y-1))$ is increasing with respect to $b$ and $y$.
The value of $b/c$ is linked to a carrying capacity, i.e. 
a value above which the population size tend to decrease. Competition destructs cells and could help for regulation
 but  its also
  make the carrying capacity decrease and at end it plays against the regulation of the trait.

\subsection{$L\log L$ criterion for branching processes with interactions}
\label{LlogL}
\cmv{Citer Athreya, et Kyprianou et al}
For branching processes, spine construction yields a conceptual approach for the Kesten Stigum criterion of non-degenerescence of the limiting martingale \cite{LPP}. For a Galton-Watson process $Z$ with reproduction r.v. $L$, $W=\lim_{n\rightarrow \infty} Z_n/\E(L)^n$ is a.s. positive on the survival event iff $\E(L\log(L))<\infty$.
The same criterion holds for  continuous time Galton-Watson, with similar approaches. We are interested in the counterpart of this criterion and approach when reproduction is density dependent. We work in the case when the original process and the spine construction do not explode.

We follow the ideas of \cite{LPP}. 
We recall that $\tau_0(z)<\infty$ for any $z\geq 1$ and  assume in this section  that for any $z\geq 1$,
\begin{equation}
\label{condetaux}
   \sum_{k\geq 1} k \tau_{k} (z)<\infty.
\end{equation}
We can thus  achieve the spine construction with $\psi=1$ and set for $z\geq 1$,
$$\lambda(z)=\sum_{k\geq 0} (k-1)\tau_{k} (z).$$
We  first get from Proposition \ref{martin} or could directly check that
$$M(t)=\exp\left(-\int_0^t \lambda(Z(s))ds\right) \,Z(t)$$
is a non-negative martingale which converges a.s. to a finite non-negative
r.v. 
$$W=\lim_{t\rightarrow \infty} M(t).$$
Similarly, we write 
$$N(t)=\exp\left(-\int_0^t \lambda(\Xi(s))ds\right) \, \Xi(t),$$
where $\Xi$ is the size of the population in the ${\bf 1}$-spine construction.
 Theorem \ref{main} yields the following expression of $\E(W)$ and a way to know when $W$ is degenerate :
\begin{lemma} Assume $\eqref{condetaux}$ and that $T_{\emph{Exp}}=+\infty$ and  $\widehat{T}_{\emph{Exp}}=+\infty$ a.s. Then, for any $z\geq 1$,
\label{lien}
$$\E_z\left(W\right)=z \, \P_z\left(\sup_{t\in [0,\infty)}  N(t)<\infty\right).$$
Besides, $\Xi-1$ is a Markov jump process  on $\mathbb N_0$ whose transition rate from $z$ to $z+k-1$  is  equal to $ k \tau_{k}( z+1)+ z\tau_{k}(z+1)$ for $z\geq 0, k\geq 0$.
\end{lemma}
The process  $\Xi-1$ can thus be seen  as the original density dependent Markov process  plus a density dependent  immigration  of $k-1$ individuals with rate
$k\, \tau_k(z)$. This   extends the result for branching processes when $\lambda$ and $\tau$ are constant and $M(t)=\exp(-\lambda t) Z(t)$.
\begin{proof} We let $t\geq 0$ and $K>0$  and  apply
Theorem \ref{main} with $\psi=1$ to function 
$$F(\mathfrak{t}, e )=F(\mathfrak{t})=\# \mathfrak g(t) e^{-\int_0^{t} \lambda(\# \mathfrak g(s))ds}\,\mathbf 1 _{\left\{\sup_{s\leq t}
\# \mathfrak g(s) . e^{-\int_0^s \lambda(\# \mathfrak g(v))dv} \, \leq K\right\}}$$
or can apply \eqref{manytoone} as well  and get
\[
\E\left(\Ind{\sup_{u\in [0,t]} M(u) \, \leq \, K}
\, M(t)\right)
=\P\left(\sup\limits_{u\in [0,t]} N(u)\leq K\right) \, .
\]
Bounded and monotone limit as $t\rightarrow \infty$ ensure
\[
\E\left(\Ind{\sup_{u\in [0,\infty)} M(u) \, \leq \, K}
\, W \right)
=\P\left(\sup\limits_{u\in [0,\infty)} N(u)\leq K\right) \, .
\]
We conclude  the proof of the first part of the proposition by monotone limit  letting $K$  go to infinity.
For the second part, we observe that $\Xi$ jumps from $z$ to $z+k-1$ with rate
$ \widehat{\tau}_{k}^{\, \star}(z)+(z-1)\widehat{\tau}_{k}(z) =k\tau_k(z)+(z-1)\tau_k(z)$.
\end{proof}
Let us derive moment conditions which guarantee that the limiting martingale is non degenerated.
These issues have already been considered, at least in the discrete framework,
motivated by controlled Galton-Watson processes \cite{K,KKR}. In these works, a monotonicity assumption or  regularity and convexity assumptions are required. Such assumptions  seem to be partially relaxed  here. Besides,
the method  can  be extended to multitype setting. The case where the process becomes critical asymptotically has received lots of attention and is often called \emph{near or almost critical}. We focus in the application here on the case where the process grows exponentially but density depend affects the growth rate. Competition can make it decrease and cooperation may make it increase, while non monotone behavior appear in particular with Allee effect. 
\begin{proposition} \label{Wcons} Assume that 
$$\inf_{z\geq 1} \lambda(z)>0, \qquad  \sum_{k\geq 1} k(\log(k)+1)\, \sup_{z\geq 1} \tau_k(z)<\infty.$$ 
 Then $T_{\emph{Exp}}=+\infty$ and  $\widehat{T}_{\emph{Exp}}=+\infty$ a.s. and for any
 for $z\geq 1$, $\E_z(W)=z.$
\end{proposition}  
The  uniformity  assumptions can be partially relaxed. For instance, with some irreducibility condition one can only assume that $\lambda$ is lower bounded by a positive constant for $z$ large enough. The $L\log L$ moment condition is necessary 
for positivity  of $\E(W)$ in some cases including  branching processes or pertubation of them. 
\begin{proof} We first notice that the fact  $\sum_{k\geq 1} k\, \sup_{z\geq 1} \tau_k(z)$
is finite provides an upperbound of the growth rate of the size of the population of
the original process  $Z$. It guarantees that $T_{\text{Exp}}=\infty$ a.s.
Let us deal with the $1$-spine construction and localize the process by
considering the stopping times $T^{m}=\inf\{ t\geq 0 : \Xi_t\geq m\}$ for $m\geq 1$.
We separate the component coming from immigration and give a trajectorial representation of $\Xi-1$.
Let us consider $V=\Xi-1$. For $t\leq T^m$,  it is defined as the unique strong solution of the following SDE
\begin{align*}
V(t)&=V(0)+\int_0^t\int_{\R_+\times \mathbb N} \Ind{u\leq V(s-)\tau_k(V(s-)+1)} \, (k-1)\, \mathcal N(ds,du,dk)\\
&\qquad \qquad+\int_0^t\int_{\R_+\times \mathbb N} \Ind{u\leq k\tau_k(V(s-)+1)} \, (k-1) \, \mathcal N_I(ds,du,dk),
\end{align*}
where we use two independent Poisson point measures, $\mathcal N$  and $\mathcal N_I$, with intensity $ds\, du \, \mathfrak n(dk)$ on $\R_+^2\times \N$, where 
 $\mathfrak n=\sum_{k\in  \N_0} \delta_k$ is the counting measure, see e.g. \cite{BM}. Defining 
$$N_1(t)=V(t)e^{-\int_0^t \lambda (\Xi(s))ds}=N(t)-e^{-\int_0^t \lambda (\Xi(s))ds},$$
we get for $t\leq T^m$,
\Bea
N_1(t)&=&N_1(0)-\int_0^t  \lambda(\Xi(s)) V(s) ds \\
&& \qquad + \int_0^t\int_{\R_+\times \mathbb N} \Ind{u\leq V(s-)\tau_k(\Xi(s-))} \,(k-1) \, e^{-\int_0^s \lambda  (\Xi(v))dv} \, \mathcal N(ds,du,dk)\\
&& \qquad +\int_0^t\int_{\R_+\times \mathbb N} \Ind{u\leq k\tau_k(\Xi(s-))} \,  (k-1)\, e^{-\int_0^s \lambda (\Xi(v))dv}\, \mathcal N_I(ds,du,dk)
\Eea
Then 
\Bea
N_1(t)&=&N_1(0)+\int_0^t\int_{\R_+\times \mathbb N} \Ind{u\leq \Xi(s-)\tau_k(\Xi(s-))}\,  (k-1)\,e \, ^{-\int_0^s \lambda(\Xi(v))dv}\,\widetilde{\mathcal{N}}(ds,du,dk)\\
&&\qquad +\int_0^t \int_{\R_+\times \mathbb N} \Ind{u\leq k\tau_k(\Xi(s-))} \, (k-1)\, e^{-\int_0^s \lambda(\Xi(v))dv}\mathcal N_I(ds,du,dk),
\Eea
where $\widetilde{\mathcal{N}}$ is the compensated measure of $\mathcal{N}$.
Thus, conditionally on  $\mathcal N_I$, $N_1(.\wedge T_m)$ is a submartingale. 
Besides, writing $c=\inf \lambda>0$ and $p_k=\sup_{z\geq 1} \tau_k(z)<\infty$, we get 
for any $m\geq 1$ and $t\geq 0$
\Bea
\E_z(N_1(t\wedge T^m) \,  \vert \, \mathcal N_I)\leq z-1+ \int_0^t 1_{u\leq k p_k} \, (k-1)\, e^{-cs} \, \mathcal N_I(ds,du,dk).
\Eea
Let us show that the $L\log L$ assumption ensures that the right hand side   is a.s. bounded with respect to $t$. Indeed 
\begin{equation}
\label{series}
\int_0^{\infty}1_{u\leq kp_k} \, (k-1) \, e^{-cs} \mathcal N_I(ds,du,dk)
=\sum_{i\geq 0} \widehat{L}_i e^{-cS_i} 
\end{equation}
is a  compound Poisson process, where $(S_{i+1}-S_i: i\geq 0)$ are  i.i.d. exponential random variables with parameters $\mu=\sum_{k\geq 2} kp_k \in [0,\infty)$ and  $(\widehat{L}_i :i\geq 0)$ are i.i.d random variables with the size bias distribution $\P(\widehat{L}=k-1)=kp_k/\mu$ for $k\geq 2$. By Borel Cantelli lemma, the fact that $\sum_{k\geq 2} \log(k) k p_k<\infty$ ensures that $\limsup_{n\rightarrow\infty} \log(\widehat{L}_n)/n=0$ p.s. Adding that $c>0$ and that $S_i$ grows linearly a.s. to infinity as $i$ tends to infinity,   the series in \eqref{series} are a.s. finite.\\
 We get then that $\Xi$ is not explosive by using 
 that $\lambda$ is upper bounded and letting $m\rightarrow \infty$.
By Fatou's lemma, we obtain that $\sup_{t\geq 0} \E_z(N_1(t) \,  \vert \, \mathcal N_I)<\infty$ a.s. Thus, the quenched submatingale $(N_1(t))_{t\geq 0}$ converges to a finite random variable a.s. as $t\rightarrow \infty$.
 So does $N(t)$, towards the same limit, since $\inf \lambda >0$.    
 Lemma \ref{lien} allows then to conclude.
\end{proof}
In particular, we can describe the growth of the process $Z$. When $\tau(z)$ tends to $b$ as $z\rightarrow \infty$ fast enough, the 
 robustness of exponential growth of Galton Watson process is expected. It  has already been studied  in the discrete setting and needs in general some  technical conditions, see the works mentioned above and also
 Klebaner \cite{Klebaner}.
\begin{corollary} Under  assumptions of Proposition \ref{Wcons},  assume further
 that $\lim_{z\rightarrow \infty} \lambda(z)=b>0$. Then
$$\lim_{t\rightarrow \infty} \log(Z_t)/t =b \qquad \text{with positive probability}.$$
Assuming further that there exists $a>1$ such that
$\vert \lambda (z)-b\vert \leq C\log(z+1)^{-a}$ for any $z>0$,
then
$$\lim_{z\rightarrow \infty}  e^{-bt} \, Z(t)= W \in (0,\infty) \quad \text{with positive probability}.$$ 
\end{corollary}
A natural question now is to know if the limiting martingale is a.s. positive on the survival event. It is well known   for branching processes and a direct consequence of the branching property. We expect extensions to similar processes with interactions. The papers mentioned above in discrete time contain interesting results in this direction. Finding relevant general conditions seems a delicate and interesting  problem. Extension to multiple dimension is also natural. In infinite dimension, for the case of branching processes, we refer to \cite{Athreya} for a similar point of view and sufficient conditions of non-degenerescence.

\begin{proof}  Using monotonicity of $Z$ 
or the previous proposition, we first observe 
 that $Z_t$ goes to infinity a.s.  as $t\rightarrow \infty$.
 Then  $\lambda(Z_t)$ tends to $b$
a.s. and the previous proposition ensures  $\lim_{t\rightarrow \infty} \log(Z_t)/t =b$ with positive probability.\\
Besides writing $r(z)=\lambda(z)-b$, 
$\int_0^{\infty} \vert r(Z_t)\vert dt<\infty$ a.s. since $\vert r(Z_t)\vert\leq C\log(\exp(bt/2)+1)^{-a}$ for $t$ large enough. It ensures that $\exp(\int_0^t \lambda(Z_s)ds)$ is a.s. equivalent to $\exp(bt)$, which ends the proof.
\end{proof}

\section{Applications to multitype processes}
\label{multitype}
Let us turn to structured populations with a finite number of types, i.e. $\# \X <\infty$.
Explicit computations of eigenelements seem to be more delicate  in general than in the single type considered above.
We consider two simple relevant regimes for population models. First,  random
but bounded population size, where conditions for existence and uniqueness of positive eigenelement are well known from Perron Frobenius theory.
Second, we  consider sampling in the  large population approximation of dynamical systems. 

\subsection{Finite  irreducible case}
\label{multifruit}
We consider  a simple case relevant for applications :  the number of types 
is finite and  the size of the population is bounded. More explicitly, we assume that
 $\#\X <\infty$ and   that there exists $\bar{z}>0$ such that 
 $$\text{ For all  }({\bf z},x,{\bf k}) \in \mathcal Z\times \mathcal X\times \mathcal Z \text{ such that } \|{{\bf z}+{\bf k}- {\bf e} (x)} \|_1 > \bar{z},\quad  \tau_{\bf k}(x,{\bf z})=0.$$
In  words, the total size of the population can not go beyond $\bar{z}$. This quantity may correspond to a carrying (or biological) capacity of the environment where population lives.
The corresponding state space with a distionguished individual 
is denoted by $\mathcal S$ defined by   
$$\mathcal S=\{ (r,{\bf v}) \in \mathcal X\times \N^{\mathcal X}:  \, {\bf v}_r\geq 1, \|{\bf v} \|_1 \leq \bar{z}\} \subset \overline{\mathcal Z}.$$ We assume that the initial condition ${\bf Z}(0)$ is a random vector of $\N^{\mathcal X}$ such that
$\|{\bf Z}(0) \|_1 \leq \bar{z}$ a.s.
We observe that boundedness ensures that the process a.s. does not explode.
We recall that $x({\bf v})$ is the finite initial population whose types are counted by ${\bf v}$
and $u_r$ a label of the population with type $r$. Besides,
 the 
 positive semigroup $M$ is defined  by
$$M_tf(a)=M_tf(r,{\bf v})=\E_{x({\bf v})}\left( \sum_{y \in \mathcal X} {\bf Z}^{(u_r)}_y(t)  f(y, {\bf Z}(t))\right)$$
for any non-negative function $f$ on $\mathcal S$ and $a=(r,{\bf v})\in\mathcal S$.
Similarly, the operator $\mathcal G$ is restricted to real functions $f$ defined on $\mathcal S$  and  for any $(x,{\bf z})\in \mathcal S$ 
\cmv{Preciser l'espace de definition de $f$ qui dans l'absolu a pour l'instant encore besoin d'aller au dela de $f$}
\Bea
\mathcal Gf(x,{\bf z})
&=& \sum_{\substack{{\bf k}\in \mathcal Z\\ \|{{\bf z}+{\bf k}- {\bf e} (x)} \|_1\leq \bar{z}}}  \tau_{{\bf k}}(x,{\bf z}) \, \langle {\bf k},f(.,{\bf z}+{\bf k}-{\bf e}(x))\rangle\\
&&+\sum_{\substack{y\in \mathcal X,{\bf k} \in \mathcal Z  \\ \|{{\bf z}+{\bf k}- {\bf e} (y) \|_1\leq \bar{z}} }}\tau_{{\bf k}}(y,{\bf z})({\bf z}_y-\delta_y^x)
f(x,{\bf z}+{\bf k}-{\bf e}(y)) -\left(\sum_{y \in \mathcal X}\tau(y,{\bf z}){\bf z}_y\right)f(x,{\bf z}).\nonumber 
\Eea
Functions on $\mathcal S$ can be identified to real vectors indexed by $\mathcal S$, which is finite.
The operator $\mathcal G$ is thus a positive linear operator on the finite dimensional space 
$\R^{\mathcal S}$ and   can be identified to a finite  square matrix.
Under irreducibility conditions,  Perron Frobenius theorem  ensures the  existence (and uniqueness up to a positive constant) of a positive eigenfunction (or eigenevector) $h$  for the semigroup
$M$ and its generator $\mathcal G$. 
Using the corresponding  $h$-spine construction,
we obtain  a characterization of the ancestral lineage (or pedigree) of a typical individual, and in particular the ancestral types. We refer to \cite{JN, GB} and references therein for similar issues for multitype branching processes and the description of ancestral lineage using the eigenelements of the first moment semigroup. This description will involve the  stationary law of the Markov process $(Y(t),{\bf \Xi}(t))_{t\geq 0}$. 
Finally, Perron Frobenius  theorem
also ensures the existence of a left eigenvector $\gamma$ of matrix $\mathcal G$ (or right eigenvector for the dual operator $\mathcal G^{\star}$), for the same eigenvalue as $h$. \\

We can  now state and prove   the result.  Let us consider $t\geq 0$ and again  a uniform choice $U(t)$ in $\G(t)$. We set for  $a\in \mathcal S$ and ${\bf k}\in \mathcal Z$,
$$P_{a}(t)=\int_0^t {\bf 1}_{(Z_{U(t)}(s),  Z_s)=a} \, ds,\qquad  N_{a,{\bf k}}(t)=\#\{ u \preccurlyeq U(t)  :  (Z_u,  {\bf Z}^u)=a, \,   {\bf K}_u={\bf k}\},$$
where $Z_{u}(s)$ is the type of the unique ancestor of $u$ at time $s$, ${\bf Z}^u$ (resp.
${\bf K}_u$)  is the type composition of the population (resp. of offsprings of individual $u$) when $u$ branches. In words, $P_a$ records the time spent in state $a$ by the ancestral lineage  and $N_{a,{\bf k}}$ the number of branching events with offsrpings ${\bf k}$.
\begin{proposition}\label{elpropre}
Assume that for any  $a,b\in \mathcal S$,  $M_1{\bf 1}_b(a)>0$.  \\
Then,  there exists a  unique triplet $(\lambda, h,\gamma)$ where 
$\lambda \in (-\infty,0]$ and $h,\gamma \in (0,\infty)^{\mathcal S}$  and
$\sum_{a\in \mathcal  S} \gamma_a=\sum_{a\in \mathcal  S} h_a \gamma_a=1$ and
 $$\mathcal G h = \lambda \, h,  \quad  \gamma \mathcal G = \lambda \, \gamma.$$
Moreover,   writing $(\mathcal A, E)$ the corresponding 
$h$-spine construction,  for any $t\geq 0$ and any measurable non-negative function $F : \mathbb T\times \mathcal U\rightarrow \R$, we have for any non-empty initial condition ${\bf x}$,
\begin{align*}
 &\E_{{\bf x}}\left( \Ind{\mathbb G(t)\ne \varnothing} \, F(\T(t), U(t))  \right) =\langle  {\bf v},h(., {\bf v})\rangle e^{\lambda t} \, \E_{{\bf x}}\left(\, \frac{1}{h(Y(t), {\bf \Xi}(t)) \, \| {\bf  \Xi}(t) \|_1} 
F(\A(t),E(t))\right).
\end{align*}
\end{proposition}
Assumption $M_1{\bf 1}_b(a)>0$ amounts to an irreducibility property of the population process ${\bf Z}$, with a distinguished particle, excluding the state when the whole population
is extincted.  Let us illustrate this condition on the following spatial model with  competition. Consider a finite number of sites with  finite carrying capacities.  On each site,  each individual gives birth to one offsprings with a positive 
 rate, when it has not reached   the carrying capacity, and dies with a positive rate. These individual rates may be dependent of the local and global density of individuals. 
 Besides, each individual may move from one site to another. This model
 satisfies the assumptions of the previous statement as soon as
 the motion of individuals  (including their offsprings) 
 is irreducible, i.e. when the graph of nodds   whose  oriented edges  correspond
 to  positive probability of transition at branching events is strongly connected.   
\begin{proof} The first point  is a direct consequence of Perron Frobenius theorem.
The fact that the  eigenvalue $\lambda$  is not positive
is due to the fact that the process  is bounded.
The second part  is then a consequence of Theorem \ref{main}, recalling that there is no explosion and that $\lambda$ is constant since $h$ is an eigenfunction.
\end{proof}
The Markov process $(Y,{\bf \Xi})$  takes values in a finite state space 
and the assumption and the positivity of $h$ ensures that it is irreducible.
We derive  the following ergodic behavior, where the limiting law does not depend on the initial (non empty) condition ${\bf x}$ (omitted in notation).
\begin{corollary} Under conditions of Proposition \ref{elpropre},
$(Y(t),{\bf \Xi}(t))$ 
 converges in law to  $\pi=(\pi_{a})_{a\in {\mathcal S}}$ as $t\rightarrow \infty$,  where   $\pi_a=h_a\gamma_a$ for  $a\in {\mathcal S}$.\\
Besides 
 for any $a=(x,{\bf z}) \in \mathcal S$ and ${\bf k}\in \mathcal Z$ such that $\|{{\bf z}+{\bf k}- {\bf e} (x)} \|_1\leq \bar{z}$,
$$\left(\frac{P_{a}(t)}{t},  \frac{N_{a,{\bf k}}(t)}{t}\right) _{\vert \G(t)\ne \varnothing} \Rightarrow \left( \pi_{a}, 
\gamma_{a}\, \tau_{{\bf k}}(a)\, \langle {\bf k}, h(.,{\bf z}+{\bf k}-{\bf e}(x))\rangle \right)$$
as $t\rightarrow \infty$, where the convergence  of the couple  holds in  law  conditionally on the event $\G(t)\ne \varnothing$.
\end{corollary}
\begin{proof} First, we recall that  the generator of $(Y,{\bf \Xi})$
is  the $h$-Doob-transform of ${\mathcal G}$, i.e.
$f\rightarrow \mathcal G(h f)/h-\lambda f$.  We can then check that 
$(h_a\gamma_a)_{a\in \mathcal S}$ is a stationary law,  using that $\gamma \mathcal G=0$. Uniqueness of stationary law holds by irreducibility and the first part is proved. \\
We consider the $h$-spine construction $(\mathcal A, E)$
and  we write for $a=(x,{\bf z}) \in \mathcal S$ and $k\in \mathcal Z$,
$$N^{\, \star}_{a,{\bf k}}(t)=\#\{ u \preccurlyeq E(t)  : \,  ({\Xi}_u,{\bf \Xi}^u)=a, \,  \widehat{{\bf K}}^{\, \star}_u={\bf k}\},$$
where $\widehat{{\bf K}}^{\, \star}_u$ is the type composition of the offsprings of the spine $u$ when it branches and ${\bf \Xi}^u$ the state of the population when it branches. Then
ergodic theorem  ensures the a.s. convergence: 
\begin{equation}
\label{psB}
\lim_{t\rightarrow \infty} \frac{N^{\, \star}_{a,{\bf k}}(t)}{t}=  \pi_{a}\, \tau_{{\bf k}}(a)\, \frac{\langle {\bf k}, h(.,{\bf z}+{\bf k}-{\bf e}(x))\rangle}{h(a)}=:\widetilde{\pi}_{a,{\bf k}}.
\end{equation}
We did not find the appropriate reference in continuous time but the proof can be achieved 
for instance by  standard renewal argument  (strong renewal theorem)  using that the successive times when a Markov jump process is in a given state and make a given jump forms a renewal process, here with finite expected mean.\\
The result is then a consequence of the previous proposition. Indeed for any $t\geq 0$ and $F$ measurable and positive, we get
$$\E\left( \mathbf 1_{\G(t)\ne \varnothing} F(N_{a,{\bf k}}(t)) \right)=e^{\lambda t}\, \langle  {\bf v},h(., {\bf v})\rangle
\E\left(\frac{1}{h(Y(t),{\bf \Xi}(t)) \|{\bf \Xi}(t) \|_1}F(N^{\, \star}_{a,{\bf k}}(t))\right)
$$
and 
$$\E\left( \mathbf 1_{\G(t)\ne \varnothing} \right)=e^{\lambda t}\, \langle  {\bf v},h(., {\bf v})\rangle
\E\left(\frac{1}{h(Y(t),{\bf \Xi}(t)) \,  \|{\bf \Xi}(t) \|_1}\right)
$$
Considering
$F(n)={\bf 1}_{\vert n /t -\widetilde{\pi}_{a,{\bf k}}\vert \geq \varepsilon}$
for $\varepsilon >0$ and using that
 $h$ and ${\bf \Xi}$ are bounded and
 taking the ratio of the two expectations, \eqref{psB}  yields 
  $$\P\left( \vert  N_{a,{\bf k}}(t)/t -\widetilde{\pi}_{a,{\bf k}}\vert \geq \varepsilon \big \vert \G(t)\ne \varnothing\right) \stackrel{t\rightarrow \infty}{\longrightarrow} 0.$$
The proof is analogous for the limit of $P_{a}(t)/t$ when $t\rightarrow\infty$.
\end{proof}
To get finer results on ancestral lineages with a spinal approach, on may be inspired from e.g. \cite{GB, CHMT, Roberts}. In particular, see \cite{GB} for a control of deviation of ancestral type frequency using large deviation theory for multitype branching processes.
Such existence and uniqueness results can be extended to infinite type space $\mathcal X$. In particular, 
Krein Rutman theorem extends  this setting  with a  compactness assumption. This result can itself be extended with pertubation of dissipative operator \cite{MS}.   Irreducibility assumption can also be coupled with Lyapunov control to obtain uniqueness of eigenelement, see \cite{BCGM} for a statement useful in our context.  That may  be the object of future interesting investigations.\\

To end this part on the finite case, let us consider a classical epidemiological
model, SIR model. In this case irreducibility fails since \emph{Recovered} is an absorbing state.
Positive eigenfunctions exist but  uniqueness does not hold.
More precisely, consider $\X=\{i,r\}$ and the Markov process
${\bf Z}=({\bf Z}_i,{\bf Z}_r)$ taking values in $\{0,\ldots, N\}^2$. The processes 
${\bf Z}_i(t)$ and ${\bf Z}_r(t)$ count respectively the number of infected and recovered individuals at time $t$ in a fixed population $N$. The branching rates are 
$$\tau_{(2,0)}(i,{\bf z})=\beta(N-({\bf z}_i+{\bf z}_r)), \qquad \tau_{(0,1)}(i,{\bf z})=\gamma,$$
where $\beta$ is the infection rate and $\gamma$ the remission rate. The other rates are $0$. 
For such an exemple, not only
the ancestral lineage of the random sample and the associated population size may be relevant for applications. When considering tracing of infected individuals, the  tree of infection associated with
the sample is involved. For this point,  the  $\psi$-construction should help. 
It is left for a future work.  We could also see a counterpart in the large population approximation in the next section.

\subsection{Large population approximation}
We consider in this section the deterministic regime appearing when the initial population is large and the process renormalized. The set of types $\X$ is still finite but the size of the population is not bounded. Our aim is to describe uniform  sampling in classical dynamical systems  for some macroscopic evolution of populations.
The scaling parameter is denoted by $N\geq 1$ and corresponds  to the order of magnitude of the size of the population, see   \cite{EKbook, Kurtz, BM} for  general references.
The space of types $\X$ is finite and the types of the initial population are given by 
$$[N{\bf v}]=
([N{\bf v}_x], \, x\in \X),$$ for some fixed positive ${\bf v}\in (0,\infty)^{\mathcal X}$. Each individual with type  $x\in \mathcal X$ living in a population  ${\bf z}\in \mathbb N_0^{\mathcal X}$ is replaced by ${\bf k}$ offsprings at rate
$$\tau^N_{{\bf k}}(x,{{\bf z}})=\tau_{{\bf k}}(x,{{\bf z}}/N),$$
where  ${\bf z}\in \R_+^{\mathcal X}\rightarrow \tau_{{\bf k}}(x,{{\bf z}})$ is a continuous function.
Let us write $${\bf x}_N=\{(u,x_u), \, u \in \mathfrak g_N\}$$
the labels and types of the initial population with type composition $[N{\bf v}]$.\\
Following the rest of the paper, we write ${\bf Z}^N$  the vector  counting types in the population and $\mathcal T^N$ the tree associated to this process. 
For sake of simplicity and regarding our motivations from population models,
we assume that
\Bea
{\bf (A1)}&&\qquad \sup_{x\in \mathcal X, \, {\bf z} \in \R_+^{\mathcal X}} \,  \sum_{{\bf k}\in \mathcal Z, \| {\bf k} \|_1 >1 } \| {\bf k}\|_1^2  \, \tau_{\bf k}(x,{\bf z}) <\infty,\\
{\bf (A2)} && \qquad  \forall K>0, \quad \sup_{x\in \mathcal X,  {\bf z}\in \mathcal Z_K} \tau(x,  {\bf z})<\infty,
\Eea 
where $\mathcal Z_K=\{ {\bf z}\in \R_+^{\mathcal X} : \|{\bf z}\|_1\leq K\}$ and $\tau(x,{\bf z})=\sum_{{\bf k}\in \mathcal Z} \tau_{{\bf k}}(x,{\bf z})$. The $\ell^2$ uniform condition in ${\bf (A1)}$ will guarantee that  the contribution
of the spine in the growth of the population size
is vanishing as $N\rightarrow \infty$. 
{\bf (A1)} and {\bf (A2)} also ensure   uniform bound on the growth rate and guarantee non explosion of the processes $Z^N$ and $\Xi^N$ for fixed $N$.  To ensure that $T_{\text{Exp}}^N=\infty$ a.s., a $\ell^1$ uniform bound in ${\bf (A1)}$ would  have been enough. We observe that these assumptions allow 
 non bounded individual death or motion rate. For instance, the
 death rate    may tend to infinity with respect to the size of the population due to competition. These assumptions also ensure that the following size dependent  growth matrix  $A({\bf z})=(A_{x,y}({\bf z}))_{x,y\in \X}$ is well defined : 
 $$A_{x,y}({\bf z})=\sum_{{\bf k}\in \mathcal Z} \tau_{{\bf k}}(x,{\bf z}){\bf k}_y\, -\tau(x,{\bf z})$$
 for ${\bf z} \in \R_+^{\X}$ and $x,y \in \X$.\\
 .

We also assume that $A$ is locally Lipschitz : 
for any $K>0$, there exists $M$ such that 
$${\bf (A3)} \qquad   \forall  x,y\in  \mathcal X, \quad \forall  {\bf z}_1,{\bf z}_2\in \mathcal Z_K, \qquad \vert A_{x,y}({\bf z}_1)-A_{x,y}({\bf z}_2) \vert \leq M \|{\bf z}_1-{\bf z}_2 \|_1.$$
Thus ${\bf z}\rightarrow {\bf z}\, A({\bf z})$ is locally Lipschitz on $\R_+^{\mathcal X}$. Using 
$({\bf A1})$ guarantees  the non explosivity of the dynamical system associated to this vector filed.     Cauchy Lipschitz theorem  then ensures the existence and uniqueness
of the solution
$({\bf z}(t,{\bf v}))_{t\geq 0}$ of the following ordinary differential equation on $\R_+$
$${\bf z}'(t,{\bf v})={\bf z}(t,{\bf v})\, A({\bf z}(t,{\bf v})), \quad {\bf z}(0,{\bf v})={\bf v}.$$

Under these assumptions, we know that ${\bf Z}^N/N$ converges in law in $\mathbb D(\R_+,\R_+^{\X})$
to the non-random process ${\bf z}(.,{\bf v})$ and refer to Theorem 2 in Chapter 11 of  \cite{EKbook}. We are actually needing  in the proof 
a counterpart for the spine construction, see below.
Finally, we  assume
 that the limiting dynamical system does not  come too close to the extinction boundary in finite time : 
 $${\bf (A4)} \qquad   \qquad  \forall T>0, \qquad  \inf_{x\in \mathcal X, t\in [0,T]} {\bf z}_x(t,{\bf v})>0. \qquad \qquad \qquad$$
 This assumption holds for many classical population models and allows us to consider 
 functions $\psi$ which go to infinity on the boundary.\\

We are interested in the limiting   $\psi$-spine construction and consider 
  a  function $\psi$  from  $\X\times [0,\infty)^\X$  to $(0,\infty)$, such that for any $x\in \mathcal X$, $\psi_x : {\bf z} \in (0,\infty)^\X \rightarrow  \psi(x,{\bf z})$ is continuously differentiable. Besides, we assume
  that   for any $\varepsilon>0$,   there exists $L$ such that
  for any  $x\in \mathcal X$ and $\mathbf z \in (\varepsilon,1/\varepsilon)^{\mathcal X}$ and $\mathbf k\in \mathbb R_+^{\mathcal X}$, 
 \begin{align}
 \label{superLip}
 \parallel  \psi(x,\mathbf z+ \mathbf k)- \psi (x,\mathbf z) \parallel_1 \leq L \parallel \mathbf k\parallel_1.
 \end{align}
The $\psi$-spine construction is initiated with a single individual, the root $E(0)=\varnothing$, whose type $Y(0)$ is chosen as follows:
$$\P(Y(0)=x)= \frac{\psi(x, {\bf v})}{\langle {\bf v},\psi(., {\bf v})\rangle} \qquad (x\in \mathcal X). $$ 
Let us explain informally why in this section the spine construction
 is  restricted    to one single initial individual. Indeed, 
 the density dependance reduces to a deterministic effect when the size of the population goes to infinity, since the normalized process  converges to the  $\mathbf z(.,\mathbf v)$.  Like for propagation of chaos, in the large population approximation, the individuals behave independently and a (time inhomogeneous) branching property holds.  
Besides,  when the limiting object $\mathbf z(.,\mathbf v)$ converges to   an equilibrium when times goes to infinity,
this non-homogeneity actually vanishes, as discussed below.
  \\

\noindent Let us be more specific.
 The spine with type $x$ branches with the following rate at time $t$
  $$\widehat{\tau}_{\bf k}^{\, \star}(x,t,{\bf v})=
  \tau_{\bf k}(x,{\bf z}(t,{\bf v}))) \frac{\langle{\bf k},\psi(.,{\bf z}(t,{\bf v}))\rangle}{ \psi(x, {\bf z}(t,{\bf v}))},$$
 while individuals with type $x$ but the spine branch at time $t$ with rate 
 $$\widehat{\tau}_{\bf k}(x,t,{\bf v})=\tau_{\bf k}(x,{\bf z}(t,{\bf v})).$$
 We use as in Section \ref{mainsection} the Ulam Harris Neveu notation to label individuals and denote by $\mathcal A_{\star}(t)$ the tree rooted in the spine. Observe also
 that ${\bf (A4)}$ and regularity of $\psi$ ensure that $\psi(.,{\bf z}(t,{\bf v}))$ is bounded on finite time intervals. Using ${\bf (A1-A2)}$ then ensures that this spine construction is not explosive. Recall that
 $E(t)$ is the label of the spine at time $t$ and  set
\cmv{c'est surtout ici ou il faut de la continuite}
 $$\mathcal Gf(x, {\bf z})= \sum_{{\bf k}\in \mathcal Z} \tau_{{\bf k}}(x, {\bf z}) \, \langle {\bf k}-{\bf e}(x),f(.,{\bf z})\rangle+\mathcal Lf_x ( {\bf z})$$
for ${\bf z} \in (0,\infty)^{\X}$ and $x,y \in \X$, where  
  $\mathcal L$  is the adjoint  operator associated to  ${\bf z} A({\bf z})$ :
$$\mathcal Lg({\bf z})= \sum_{y,x\in \mathcal X,{\bf k}\in \mathcal Z}  {\bf z}_y\tau_{{\bf k}}(y, {\bf z})  ({\bf k}_x-\delta_x^y)\frac{\partial g}{\partial {\bf z}_x}({\bf z}),$$
where $\delta_x^y=1$ if $y=x$ and $0$ otherwise.
Using ${\bf (A1)}$ and differentiability of $\psi$, $\psi$ is in the domain of $\mathcal G$ and
  we define  $\lambda$ as
$$\lambda(x,{\bf z}) =\frac{\mathcal G\psi(x, {\bf z})}{\psi(x,{\bf z})}$$
for $x\in \mathcal X$ and ${\bf z}\in \R_+^{\mathcal X}$ and can state the result on the subtree containing the sample. More precisely, recall that  $L^N_v$ is the life length of individual $v$ in the original process
$Z^N$, $L^N_v(t)$ this life length when the process is stopped at time $t$, and  
$Z^N_v$ the type of individual $v$. 
Writing $u_0$ the ancestor of $u$ at time $0$, we set
$$\T_u^N(t)=\{ (v, L^N_v(t), Z^N_v)   :  \exists s \leq t,  (u_0,v) \in   \mathbb G^N(s)\}.$$
where $ \mathbb G^N(s)$ is the set of labels alive in $\mathcal T^N$ at time $s$.
The random tree $\T_u^N(t)$ is the  tree associated  with the ancestral lineage of $u$ and their descendants, rooted 
in $\varnothing$. We endow the space $\mathbb T\times \mathcal X$
with a $\ell_1$  topology  on the collection of labels together with their life lengths and types, defined as follows. Recall that a finite tree   $\mathfrak t=\{(v,\ell_v,z_v) :  v \in \mathcal U(\mathfrak t)\}$ of $\mathbb T$ is a collection $\mathcal U(\mathfrak t)\subset \mathcal U$ of labels corresponding to individuals $v\in \mathcal U(\mathfrak t)$
of the population with time lenght $\ell_v$ and type $z_v$.
For  
two  trees $\mathfrak t=\{(v,\ell_v,z_v) :  v\in  \mathcal U(\mathfrak t)\}$ and $\mathfrak t'=\{(v,\ell'_v,z'_v) :  v \in \mathcal U(\mathfrak t')\}$. 
We write $\mathfrak t \Delta \mathfrak t':=\mathcal U(\mathfrak t)\, \Delta  \, \mathcal U(\mathfrak t')$ the set of labels of $\mathcal U$ in one tree but not in the other and $\mathfrak t \cap \mathfrak t' :=\mathcal U(\mathfrak t)\cap \mathcal U(\mathfrak t')$ the set of labels  in both.
 We consider the following  distance on trees
$$d(\mathfrak t, \mathfrak t')=\# (\mathfrak t \Delta \mathfrak t' )+\sum_{u \in \mathfrak t \cap \mathfrak t' } (\vert \ell_u-\ell_u'\vert + \vert {\bf k}_u-{\bf k}_u'\vert)$$
and endow $\mathbb T$ with this distance and $\mathbb T\times \mathcal X$ with the product topology.
\begin{proposition} \label{limsystdyn} Assume that {\bf (A1-2-3-4)} hold.
 Let $t\geq 0$ and $U^N(t)$ be a uniform choice among individuals of
 $\T^N(t)$ alive at time $t$. Then for any $F$ continuous and positive from $\mathbb T\times \mathcal X$ to $\R_+$,
$$\lim_{N\rightarrow \infty} \E_{{\bf x}_N}\left(F(\T_{U^N(t)}^N(t),U^N(t)) \right)=\, \E\left(\frac{\exp\left(\int_0^t \lambda (Y(s),{\bf z}(s,{\bf v}))\, ds\right)}{{\psi(Y(t),{\bf z}(t,{\bf v})) \, \| {\bf z}(t,{\bf v}) \|_1} }\, F(\mathcal A_{\star}(t),E(t) )\right).$$
\end{proposition}
This result  can be extended to finite multiple sampling  at time $t$ with independent construction started at initial time. Indeed, in this large population approximation and finite time horizon, the different samples at time $t$ come from different original individuals and  behave independently.
We can more generally consider  a finite number of initial individuals in the description.  Considering an infinite number of initial individuals should
lead  to change the topology for convergence. 
Besides, relaxing the $\ell^2$ uniform bound of ${(\bf A1)}$ should be interesting. Keeping 
the $\ell^1$ uniform bound would give a continuous limiting population process with potential infinite rate of branching along the spine (and the uniform sampling). Considering even larger jumps would give a stochastic limit and more complex spinal constructions. It is another interesting direction. 
\\

Let us prepare the proof of Proposition \ref{limsystdyn}.
Following Section \ref{mainsection}, we write
$(\mathcal A^N,E^N)$ the $\psi_N$-spine construction  associated to $\T^N$, with 
$$\psi_N(x,{\bf z})=\psi(x,{\bf z}/N)$$
for $x\in \mathcal X$ and ${\bf z}\in (0,\infty)^{\mathcal X}$ and initial condition
${\bf x}_N$.  Function $\psi_N$ is extended to the  space $\mathcal X\times \R_+^{\mathcal X}$
by setting $\psi_N=1$ on the boundary of $\mathcal X\times \R_+^{\mathcal X}$. We introduce 
$$\lambda^N(x,{\bf z}) =\frac{\mathcal G^N\psi_N(x, {\bf z})}{\psi_N(x,{\bf z})}$$
on $ \mathcal X \times \R_+^{\mathcal X}$, where
$$\mathcal L^N g({\bf z})=\sum_{y\in \mathcal X,{\bf k} \in \mathcal Z} {\bf z}_y\tau_{{\bf k}}(y,{\bf z}/N)\left( g({\bf z}+ {\bf k}-{\bf e}(y))
-g({\bf z})\right)$$
and
\Bea
\mathcal G^Nf(x,{\bf z})
&=& \sum_{{\bf k}\in \mathcal Z} \tau_{{\bf k}}(x,{\bf z}/N) \, \langle {\bf k}-{\bf e}(x),f(.,{\bf z}+{\bf k}-{\bf e}(x))\rangle+\mathcal L^Nf_x ({\bf z}).
\Eea
Theorem \ref{main} yields
\begin{align}
 &\E_{{\bf x}_N}\left( \Ind{ \mathbb G^N(t)\ne \varnothing} \, F(\T^N(t), U^N(t))  \right)=\nonumber\\
 &\qquad \qquad \langle  [N{\bf v}]/N,\psi(., [N{\bf v}]/N)\rangle \, \E_{{\bf x}_N}\left(G_N(\A^N(t),E^N(t))\right),\label{idpourcv}
\end{align}
where
\begin{align*}
G_N(\A^N(t),E^N(t))&=\frac{e^{\int_0^t \lambda^N(Y^N(s),{\bf \Xi}^N(s))ds}}{\psi(Y^N(t), {\bf \Xi}^N(t)/N) \, \| {\bf  \Xi}^N(t) /N\|_1} \, F(\A^N(t),E^N(t)).
\end{align*}
Roughly  speaking, all the quantities involved converge as $N\rightarrow \infty$.  
The process ${\bf \Xi}^N$ which counts the types of individuals in the $\psi_N$-spine construction converges to the same limit  as ${\bf Z}^N$. Indeed, 
when $N$ goes to infinity, Assumption ({\bf A1}) guarantees that there is no jump of order $N$ and the regularity of  $\psi$  ensures that 
$$\lim_{N\rightarrow \infty} \frac{\psi(x, ([N{\bf z}]-{\bf k}+1)/N)}{\psi(x, [N{\bf z}]/N)}=1.$$
Thus the contribution of the spine vanishes in the large population limit, despite the biased rate.  Besides, at a macroscopic level,  the other individuals behave as in the original process. We can now turn to the proof.
\cmv{Ecrire les $q$ et les $\beta$}
\begin{proof}[Proof of Proposition \ref{limsystdyn}] 
First, following  the proof of Theorem 2 in Chapter 11 of  \cite{EK}, we obtain that
 the sequence of process  $({\bf \Xi}^N)_N$ converges in law in $\mathbb D(\R_+,\R_+^{\X})$
to $({\bf z}(t,{\bf v}))_{t\geq 0}$ as $N$ tends to infinity. To adapt the proof, 
we note that ${\bf \Xi}^N$  alone is not a Markov process. One has to consider the couple  $(Y^N,{\bf \Xi}^N)_N$
but the influence of the type of the spine $Y^N$ is vanishing in computations using $\ell^2$ bound  $\bf (A1)$ and the fact the population is renormalized by $N$.  Assumptions {\bf (A1,A2,A3)} thus allow us
to get the counterpart of conditions  $(2.6),(2.7),(2.8)$ of  Theorem 2 in Chapter 11 of  \cite{EKbook}, while the initial condition converges in law by definition of the model.  \\

Now, we check that $(x,{\bf z})\rightarrow \lambda_N(x,N{\bf z})$ converges uniformly on compact sets of $\X\times (0,\infty)^{\X}$ and use a localization procedure to get the convergence in $\eqref{idpourcv}$ as $N\rightarrow\infty$. Indeed,
$$\mathcal L^N (\psi_N)_x(N{\bf z})=N\sum_{y\in \mathcal X,{\bf k} \in \mathcal Z} {\bf z}_y\tau_{{\bf k}}(y,{\bf z})\left( \psi(x,{\bf z} + ({\bf k}-{\bf e}(y))/N)
-\psi(x,{\bf z})\right).$$
Since $\psi_x$ is continuously differentiable on $(0,\infty)^{\mathcal X}$ and using ${\bf (A1)-(A2)}$, 
\begin{align*}
N\sum_{y\in \mathcal X,\|{\bf k} \|_1 \leq \sqrt{N} } {\bf z}_y\tau_{{\bf k}}(y,{\bf z})
\bigg\vert \psi\left(x,{\bf z} + \frac{{\bf k}-{\bf e}(y)}{N}\right)-\psi(x,{\bf z})-\sum_{y'\in \mathcal X} \frac{{\bf k}_{y'}-\delta_y^{y'}}{N} \frac{\partial \psi_x}{\partial {\bf z}_{y'}}({\bf z}) \bigg\vert 
\end{align*}
tends to $0$ as $N\rightarrow \infty$, uniformly for ${\bf z}\in (\varepsilon,1/\varepsilon)^{\mathcal X}$, where $\varepsilon \in (0,1)$ is fixed. 
Besides,  using \eqref{superLip} and  ({\bf A1}),
\begin{align*}
&N\sum_{y\in \mathcal X,\|{\bf k} \|_1 > \sqrt{N}} {\bf z}_y
\tau_{{\bf k}}(y,{\bf z})\bigg\vert \psi\left(x,{\bf z} + \frac{{\bf k}-{\bf e}(y)}{N}\right)-\psi(x,{\bf z})\bigg\vert\\
&\qquad \qquad  \qquad \qquad  \qquad \qquad \qquad   \leq L \varepsilon^{-1} \sum_{y\in \mathcal X, \|{\bf k} \|_1 > \sqrt{N}} (\|{\bf k} \|_1+1) \tau_{{\bf k}}(y,{\bf z})
\stackrel{N\rightarrow \infty}{\longrightarrow} 0,
\end{align*}
uniformly for ${\bf z}\in (\varepsilon,1/\varepsilon)^{\mathcal X}$. Recalling
the definition of $\lambda$ and controlling the terms for $ \|{\bf k} \|_1 > \sqrt{N}$ in $\lambda$
as above with $({\bf A1})$ ensures that
 for any $\varepsilon>0$,
$$\sup_{x\in \mathcal X ,{\bf z}\in (\varepsilon,1/\varepsilon)^{\mathcal X}}
\vert  \lambda^N(x,N{\bf z})-\lambda(x,{\bf z})\vert \stackrel{N\rightarrow \infty}{\longrightarrow} 0.$$
Using the convergence of ${\bf \Xi}^N$ to ${\bf z}(.,{\bf v})$  in  $\mathbb D(\R_+,\R_+^{\X})$ and $({\bf A4})$, 
\eqref{idpourcv} yields
\begin{align*}
& \lim_{N\rightarrow \infty} \big\vert
 \E_{{\bf x}_N}\left( \Ind{ \mathbb G^N(t)\ne \varnothing} \, F(\T^N_{U^N(t)}(t), U^N(t))  \right)
\\
&\qquad \qquad  \qquad -  \langle  [N{\bf v}]/N,\psi(., [N{\bf v}]/N)\rangle \, \E_{{\bf x}_N}\left(H(\A^N(t),E^N(t))\right)\big\vert=0,
\end{align*}
for $F$ continuous,  positive and bounded, where 
\begin{align*}
H(\A^N(t),E^N(t))&=\frac{e^{\int_0^t \lambda(Y^N(s),{\bf z}(t,{\bf v})))ds}}{\psi(Y^N(t), {\bf z}(t,{\bf v})) \, \| {\bf z}(t,{\bf v}) \|_1} \, F(\A^N_{\star}(t),E^N(t))
\end{align*}
and $\A^N_{\star}$ is the tree $\A^N$
where we only keep the tree rooted
in the initial spine individual. The conclusion can be achieved by a coupling argument, 
since the first time when one individual
of $\A^N_{\star}$ has an offspring of size greater than $\sqrt{N}$ tends to infinity. Thus 
the individual branching  rates of $\A^N_{\star}$ converge uniformly to the rates
of $\A_{\star}$, using the same  localization as above to keep
 the process ${\bf \Xi}^N$ in compact sets excluding boundaries.
\end{proof}


In general and as in the previous subsection, one may expect to solve the limit eigenproblem :
\begin{align*}
&\sum_{{\bf k}\in \mathcal Z} \tau_{{\bf k}}(x, {\bf z}) \, \langle {\bf k}-{\bf e}(x),\psi(.,{\bf z})\rangle\\
&\qquad \qquad+ \sum_{y,x\in \mathcal X,{\bf k}\in \mathcal Z}  {\bf z}_y\tau_{{\bf k}}(y, {\bf z})  ({\bf k}_x-\delta_x^y)\frac{\partial \psi_x}{\partial {\bf z}_x}({\bf z})=\lambda (x,{\bf z})\psi(x,{\bf z})
\end{align*}
for any $x\in \X$ and ${\bf z}\in \R_+^{\X}$ such that ${\bf z}_x>0$.
One also expects that uniqueness of positive normalized solution holds under irreducibility conditions.
We only illustrate the result with two simple and more explicit examples.
In one dimension $\mathcal X=\{x_1\}$, taking $\psi(z)=1/z$ is reminiscent  from the previous section for single type models.  It yields $\lambda=0$
and as $N\rightarrow \infty$, 
$(\T_{U(t)}^N(t),U^N(t))$ initiated in ${\bf x}_N$ converges in law 
to $\left(\mathcal A(t),E(t) )\right)$ as $N\rightarrow \infty$.\\
\cmv{Un cas non equilibre dim 2 ?}
Second, when the population process is at equilibrium, we can also be more explicit.
More precisely, assume that  there exists ${\bf z}_*\in\R_+^{\X}$ such that 
$${\bf z}_*A({\bf z}_*)=0.$$ 
Then ${\mathcal L}f_x ({\bf  z}_*)=0$ for any $f$ and $x\in \X$. 
The spectral problem $\mathcal G\psi=0$ simplifies since the influence  of
 the population on the spinal tree is constant. The solution of the problem is then given by $\psi(x,{\bf z})=\varphi(x)$
where $\varphi : \X\rightarrow (0,\infty) $ is solution of
$$\forall x\in  \mathcal X, \qquad \sum_{{\bf k}\in \mathcal Z} \tau_{{\bf k},*}(x) \, \langle {\bf k}-{\bf e}(x),\varphi\rangle=0,$$
and
$$ \tau_{{\bf k},*}(x)=  
  \tau_{\bf k}(x,{\bf z}_*) \frac{\langle {\bf k},\varphi(.) \rangle}{ \varphi(x)}.$$
    It means that
$$\forall x\in  \mathcal X, \qquad \sum_{y\in \mathcal X} \varphi(y) A_{y,x}({\bf z}_*)=
0.$$
Existence and uniqueness of positive $\phi$ under irreducibility assumption
is then again a consequence of Perron Frobenius theorem and we recover 
in that case  the spine construction  for
critical multitype Galton Watson process proposed in \cite{KLPP, GB}. In this vein, let us refer to \cite{CHMT}, for a more complex model in infinite dimension motivated by
 adaptation to environmental change, which uses the branching limiting structure and also describes the  backward process appearing 
 in sampling.\\

{\bf Acknowledgement.}
This work  was partially funded by the Chair "Mod\'elisation Math\'ematique et Biodiversit\'e" of VEOLIA-Ecole Polytechnique-MNHN-F.X and ANR ABIM 16-CE40-0001. The idea of this project started during IX Escuela de Probabilidad y Procesos Estoc\'asticos, in Mexico in 2018. The author is grateful to the organizers for the invitation. The author is  very grateful to Simon Harris, Andreas Kyprianou and Bastien Mallein for stimulating discussions on this topic and in particular on the link with $h$-transform. The authors also thanks
Bertrand Cloez, Pierre Gabriel, Aline Marguet, Sylvie M\'el\'eard for related discussions and motivations.

\end{document}